\documentclass[letterpaper,11pt]{article}

\usepackage{ucs}
\usepackage[utf8x]{inputenc}
\usepackage{graphicx}
\usepackage{amsfonts}
\usepackage{dsfont}
\usepackage{amssymb}
\usepackage{amsmath}
\usepackage{amsthm}
\usepackage{enumerate}
\usepackage{stmaryrd}
\usepackage{fullpage}
\usepackage{ifthen}
\usepackage{subfigure}
\usepackage{epic}
\usepackage{authblk}
\usepackage{textcomp}
\usepackage[small]{caption}
\usepackage{bm}

\usepackage[numbers,comma,square,sort&compress]{natbib}
\usepackage[letterpaper,text={7in,9in},centering]{geometry}

\usepackage{color}
\usepackage{titlesec}
\graphicspath{{eps/}{pdf/}}

\newcommand{\be}{\begin{equation}}
\newcommand{\ee}{\end{equation}}
\newcommand{\bqs}{\begin{equation*}}
\newcommand{\eqs}{\end{equation*}}

\newcommand{\bg}{{\bar{\gamma}}}

\newcommand{\slin}{s_{\mathrm{lin}}}

\renewcommand{\O}{\mathcal{O}}

\newcommand{\F}{\mathcal{F}}

\numberwithin{equation}{section}

  {\begin{trivlist}\item[]{\bf Hypothesis #1 }\em}{\end{trivlist}}

\theoremstyle{plain}
\newtheorem{theorem}{Theorem}[section]

\newtheorem{lemma}[theorem]{Lemma}

\newtheorem{rmk}[theorem]{Remark}

\newcommand{\mbi}{\mathbf{i}}

\newenvironment{Proof}[1][.]%
 {\begin{trivlist}\item[]\textbf{Proof#1 }}%
 {\hspace*{\fill}$\rule{0.3\baselineskip}{0.35\baselineskip}$\end{trivlist}}

\title{Locked fronts in a discrete time discrete space population model}

\author{Matt Holzer\footnote{email: \texttt{mholzer@gmu.edu}}, Zachary Richey\footnote{email: \texttt{zrichey@gmu.edu}}, Wyatt Rush\footnote{email: \texttt{wrush@gmu.edu}}, Samuel Schmidgall\footnote{email: \texttt{sschmidg@gmu.edu}} }
\affil{\small Department of Mathematical Sciences, George Mason University, Fairfax, VA, USA}

\begin{document}
\maketitle

\begin{abstract}

A model of population growth and dispersal is considered where the spatial habitat is a lattice and reproduction occurs generationally.  The resulting discrete dynamical systems exhibits velocity locking, where rational speed invasion fronts are observed to persist as parameters are varied.  In this article, we construct locked fronts for a particular piecewise linear reproduction function.  These fronts are shown to be linear combinations of exponentially decaying solutions to the linear system near the unstable state.  Based upon these front solutions, we then derive expressions for the boundary of locking regions in parameter space.  We obtain leading order expansions for the locking regions in the limit as the migration parameter tends to zero.  Strict spectral stability in exponentially weighted spaces is also established.

\end{abstract}



\section{Introduction}
We study a model of population dynamics introduced in \cite{wang19}, where both space and time are discrete quantities.  To envision the model, imagine an infinite chain of islands and a species of bird.  Suppose that this species initially resides on a single island in the chain.  During each generation, both migration and reproduction occur.  First, some proportion of the bird population migrates to neighboring islands while the rest remains.  Second, the population at each island  reproduces independently  according to some reproduction rule.  Repeating this process over many generations, the species spreads out and forms a traveling front.  The speed of this front characterizes how quickly the island chain is populated by the new species, and of interest is how this speed depends on system parameters.  For example, one might imagine that a small increase in the migration rate would lead to a faster invasion speed.  However, as was noted in \cite{wang19}, this is not always the case, and for some reproduction functions and some parameters, the invasion speed can be {\em locked} and remain constant over some subset of parameter space.  This locking phenomena is the primary focus of this article, and our primary goal is to construct locked traveling fronts and determine conditions that prescribe the set of parameters over which these fronts exist.

We now describe the mathematical formulation of the model introduced in \cite{wang19}.  Let $u_{i,t}$ be the population at the $i$-th lattice site during the $t$-th generation.  Following the description above, each generation consists of two steps: migration and reproduction.  First, it is assumed that some proportion $m$ of the population at each lattice site will migrate, with half moving left and the other half moving right.  A reproduction function $g(u)$ then prescribes the population in the next generation as a function of the post-migration population at each island.  Putting these two steps together, we have the following difference equation
\be u_{i,t+1}=g\left( \frac{m}{2}u_{i-1,t}+(1-m)u_{i,t}+\frac{m}{2}u_{i+1,t}\right).\label{eq:main} \ee
A variety of reproduction functions were considered in \cite{wang19}.  Here, we will focus on the most analytically tractable case, namely
\be g(u)=\left\{ \begin{array}{cc} r u & 0\leq u<c \\ 1 & u\geq c \end{array}\right. . \label{eq:g} \ee
We only consider the case where $rc \le 1$, that is, $g(u) \le 1$ for any $u\geq 0$.  The parameter $c$ represents a critical population density.  Below this threshold, the reproduction function is linear with a proportionality constant $r$.  Above this threshold, the reproduction function returns the value of $1$, which is the carrying capacity of the lattice site.  This jump in the reproduction function is characteristic of a weak Allee effect, where the maximal per capita growth rate occurs at intermediate values of the population density.

Invasion speeds determined from direct numerical simulations for two different sets of parameters are shown in Figure~\ref{fig:speedplots}.  These speeds are numerically calculated as the ratio of the number of lattice sites (to the right) that transition to the carrying capacity divided by the number of generations simulated.  When the critical threshold $c$ is large, the invasion is dominated by the linear growth ahead of the front interface, and the selected invasion speed appears to be a smooth monotonically increasing function of the migration rate $m$.  By contrast, for smaller values of $c$, it is observed that velocity locking can occur, where the speed of the front remains fixed over an interval of parameter values.  As is described in \cite{wang19}, this locking is a consequence of the discrete nature of the problem.  Fronts traveling with rational speed are fixed points of a certain map: 
for rational speed $s=\frac{p}{q}$,  this map consists of $q$ fold iteration of (\ref{eq:main}), followed by shifting the solution $p$ lattice sites to the left.  In the case of locking, these fixed points are robust with respect to small changes in parameters, leading to preservation of the front over an interval of parameter values.   The speed plot in the right panel of Figure~\ref{fig:speedplots} resembles a Devil's staircase and suggests an analogy to phase locking; see for example \cite{arnold61}.  Indeed, in parameter space the locking regions resemble resonance tongues; see Figure~\ref{fig:tonguespic}.  

\begin{figure}[!t]
\centering
\includegraphics[width=0.45\textwidth]{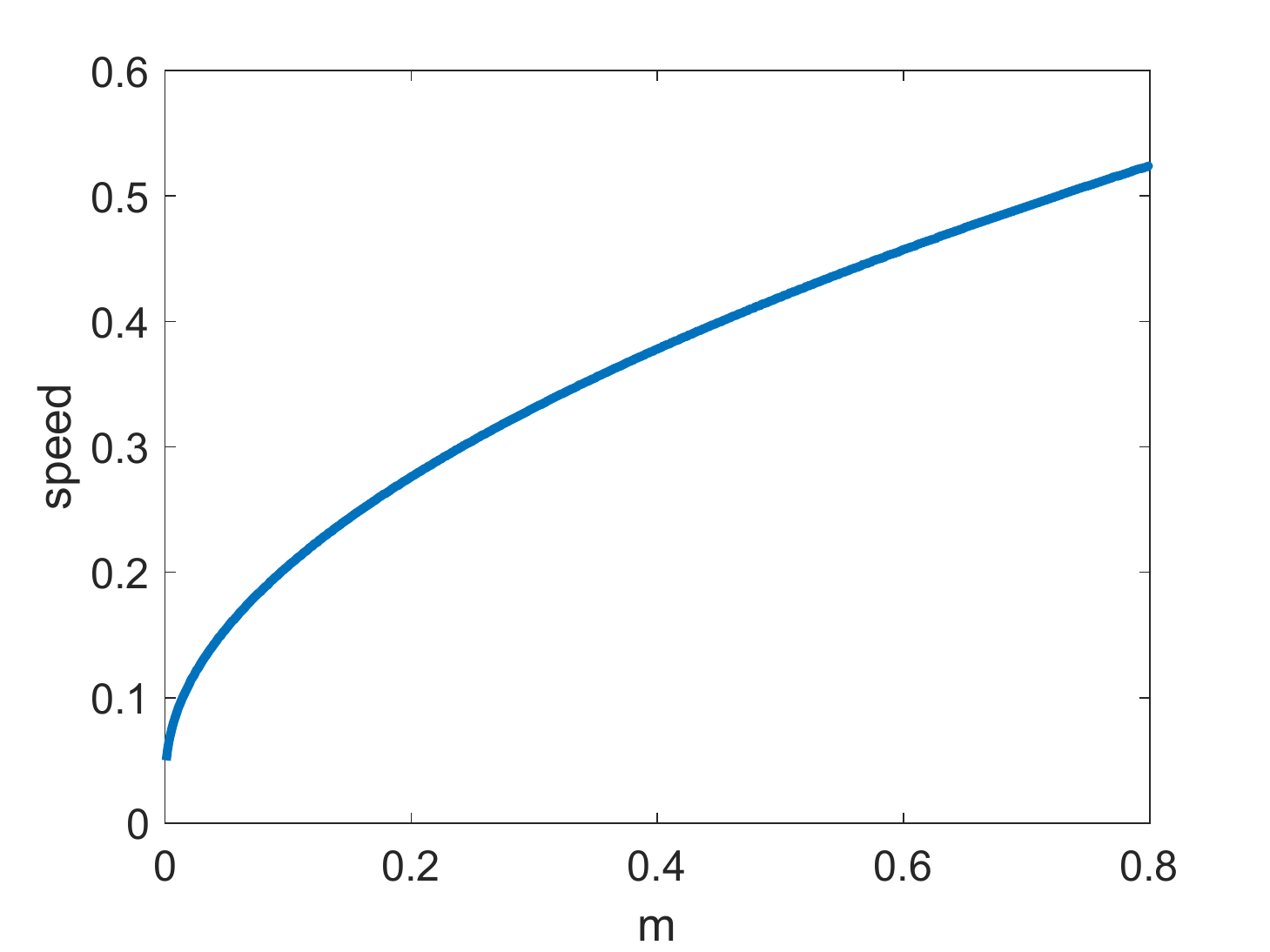}
\includegraphics[width=0.45\textwidth]{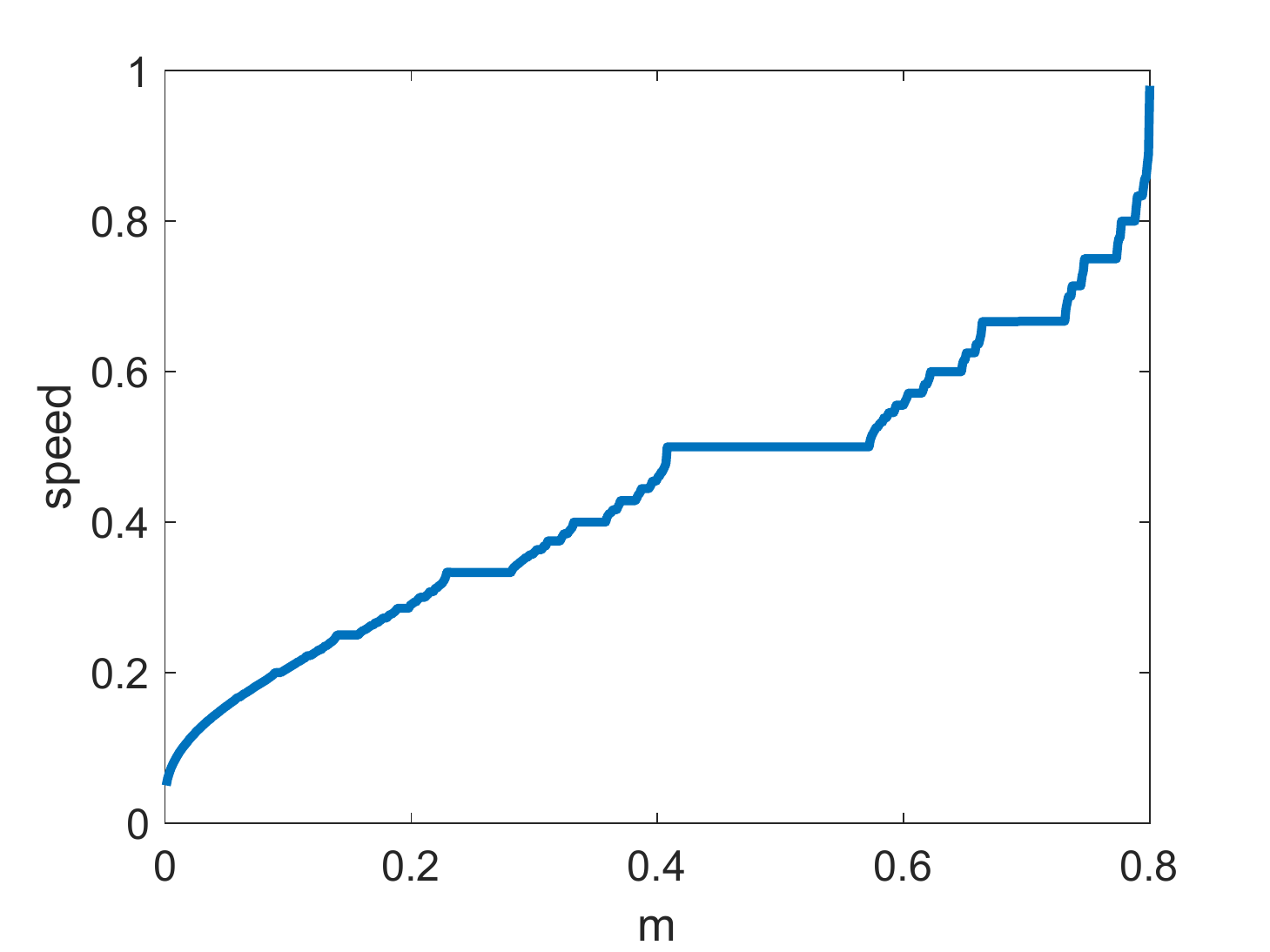}
\caption{Numerically observed invasion speeds for (\ref{eq:main}) as a function of the migration rate $m$ with $r=1.2$ and two different choices of the critical population density $c$.  On the left, the case of $c=0.8$ is depicted, and the invasion speed appears to be a smooth monotonically increasing function of the migration rate.  On the right, the case of $c=0.4$ is depicted for which the invasion speed appears to be constant at certain rational speeds and resembles a Devil's staircase.     }
\label{fig:speedplots}
\end{figure}

\begin{figure}[!t]
\centering
\includegraphics[width=0.45\textwidth]{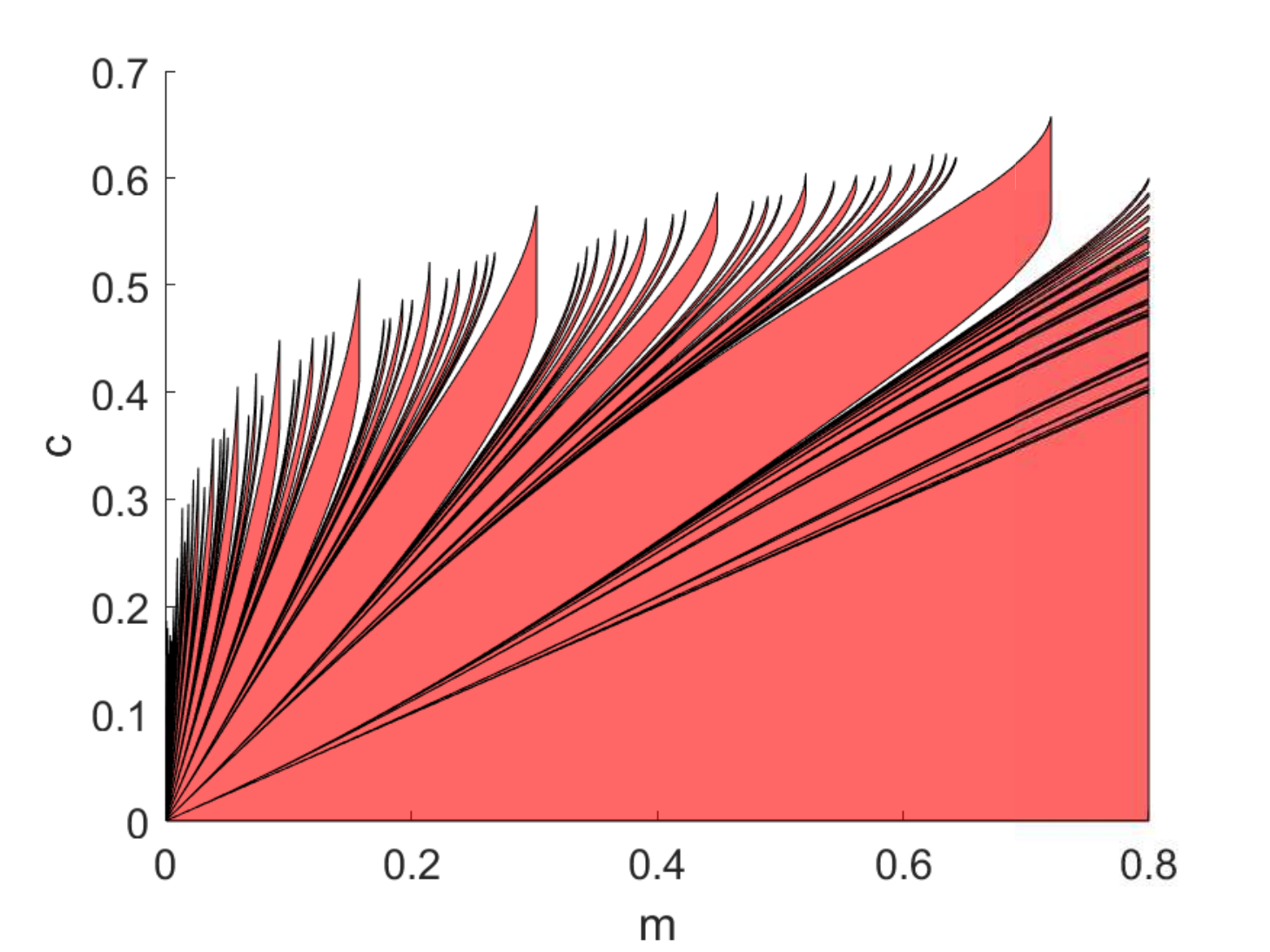}
\includegraphics[width=0.45\textwidth]{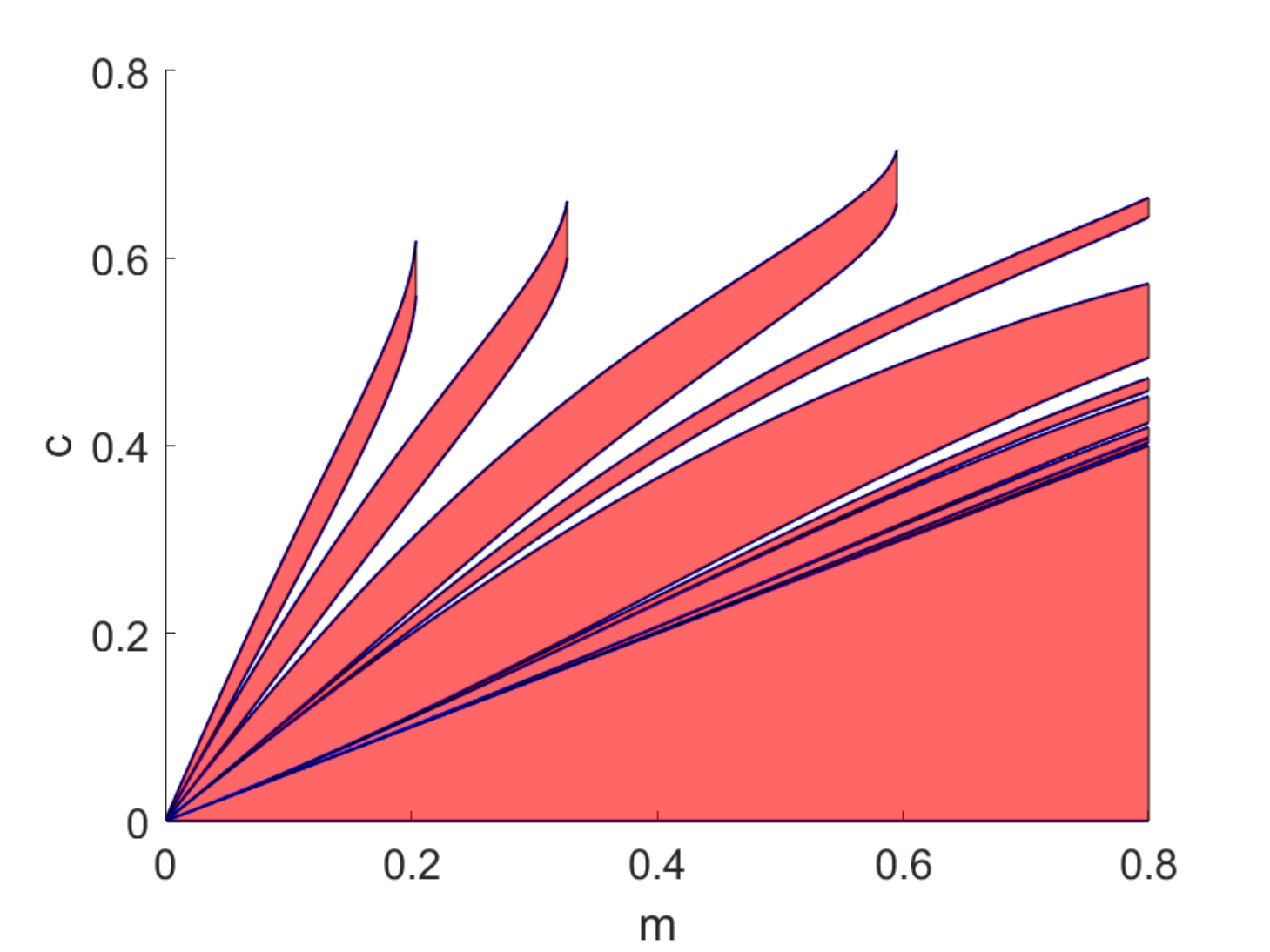}
\caption{Locking regions (shaded) as a subset of $c$-$m$ parameter space with $r=1.2$.  Shown are regions for  all rational speeds $\frac{p}{q}$ with $q\leq 20$ and $1\leq p\leq q$ with $\mathrm{gcd}(p,q)=1$. These regions are obtained via the formulas derived in Section~\ref{sec:lockedfronts}.   }
\label{fig:tonguespic}
\end{figure}

Fronts propagating into unstable states have been studied extensively; see for example \cite{vansaarloos03}.  Most investigations involve PDE models where both space and time are continuous variables.  In this context, invasion fronts can be characterized as {\em pulled} if their speed is equal to the spreading speed of disturbances for the equation linearized near the unstable state, and {\em pushed} if their speed is determined by nonlinear effects.  When space is discrete, the same dichotomy exists, and it is only in the case of both discrete time and space that velocity locking is observed.  In \cite{wang19}, {\em locked} fronts are introduced as a subset of pushed fronts where the rational velocity is constant over some region in parameter space.  

Velocity locking for traveling fronts has also been studied for difference equations known as coupled map lattices, where the fronts connect two stable states; see for example \cite{kaneko86}.  In some cases, the dynamics are shown to be equivalent to a circle map, and an explicit analogy to mode locking is achieved; see for example \cite{carretero00,fernandez97}.  Velocity locking with zero speed is also known as front pinning and has been studied widely in the literature.  In contrast to the velocity locking studied here, front pinning can occur for a variety of systems with both time and space as continuous variables for reaction functions of bistable type.  Pinning was originally studied in the context of one-dimensional lattices; see for example  \cite{bell84,carpio03,fath98,keener87}.  Pinning was later discovered to occur due to heterogeneities; see for example \cite{dirr06,lewis00,xin93},  in nonlocal equations; see for example \cite{anderson16,bates97}, and for problems posed in higher spatial dimensions; see for example \cite{berestycki16,hoffman10,lewis91}

The primary contribution of the current study is to construct locked fronts for (\ref{eq:main}) and derive boundaries of the locking regimes in parameter space.  In general, construction of traveling waves for lattice dynamical systems is challenging.  Take for example a front propagating with rational speed $s=\frac{p}{q}$.  After $q$ generations, the population at any lattice site will depend on the population at $2q+1$ lattice sites in original generation.  This can be re-expressed in the form of a traveling wave equation as a dynamical system in $\mathbb{R}^{2q}$.  Further complicating the matter is that unless $g(u)$ has an analytical inverse, this dynamical system is defined implicitly.   Constructing solutions in such a high dimensional phase space is an extremely challenging problem.  By restricting to the piecewise linear reproduction function in (\ref{eq:g}), this construction becomes tractable by allowing us to piece together linear solutions near zero with the stable state one.  

Our main result is presented in Theorem~\ref{thm:main}.  There we show that for any rational $s=\frac{p}{q}<1$ there exists a nonempty region in $(r,m,c)$ parameter space for which positive fronts of (\ref{eq:main}) propagating with speed $s=\frac{p}{q}$ exist.   These fronts are fixed point of the following map.  Let $G:\ell^\infty(\mathbb{Z})\to \ell^\infty(\mathbb{Z})$ be the generational map defined in (\ref{eq:main}) and let $S:\ell^\infty(\mathbb{Z})\to \ell^\infty(\mathbb{Z})$ be the left shift operator.  Locked fronts are fixed points of the map $\mathcal{F}(u)=S^{(p)}\left(G^{(q)}(u)\right)$.  Our second main result concerns the stability of the locked fronts as a fixed point of $\mathcal{F}$.  In Theorem~\ref{thm:stable}, we study the spectral stability of this solution and show that its spectrum in an appropriate weighted space is strictly contained inside the unit circle.  


The rest of this paper is organized as follows.  In Section~\ref{sec:outline}, we provide a short outline of our approach.  In Section~\ref{sec:prelim}, we derive some preliminary facts about (\ref{eq:main}) linearized near the unstable equilibrium.  In Section~\ref{sec:lockedfronts}, we construct locked fronts propagating with rational speed and state our main Theorem~\ref{thm:main} prescribing the existence of locking regions.  Portions of the proof are presented in Section~\ref{sec:positive} and we derive expansions for the locking regions.  In Section~\ref{sec:stab} we prove that the front is spectrally stable with respect to perturbations in a particular weighted function space; see Theorem~\ref{thm:stable}.  In Section~\ref{sec:numerics}, we compare our predictions to numerical simulations.  Finally, we conclude in Section~\ref{sec:discussion} with a discussion of future directions for study.

\section{Front Construction: Overview}\label{sec:outline} 

Let us  motivate the construction that will follow.  Locked fronts propagating to the right with speed $\frac{p}{q}$ are solutions of (\ref{eq:main}) which return to the same form after $q$ generations but are shifted $p$ lattice sites to the right.  For example, consider the following example of a speed $\frac{2}{5}$ front initially located at lattice site $i=0$ and evolving over five generations:
\[\begin{array}{ccccccc}
\text{Lattice Site} & i=-1 &  i=0 & i=1 & i=2 & i=3 & i=4 \\
\text{Generation 0 } & 1& 1 & \phi_1  & \phi_2 & \phi_3 & \phi_4 \\
\text{Generation 1 } & 1&1 & *  & * & * & * \\
\text{Generation 2 } & 1&1 & * & * & * & * \\
\text{Generation 3 } & 1&1 & 1 & * & * & *  \\
\text{Generation 4 } & 1&1 & 1 & * & *  & * \\
\text{Generation 5 } & 1&1 & 1 & 1  & \phi_1 & \phi_2  
\end{array} \]
Our goal is to compute the $\phi_i$ that describe the front as well as the front profile during intermediate generations, marked in the table with asterisks.  We make several observations that will guide our approach in the coming sections.  We say that a lattice site is \textit{at capacity} if the population is one at that lattice site.  Lattice sites to the left of the front interface are at capacity and remain at capacity. For those lattice sites ahead of the front interface, the update rule is linear.  As a result, we expect that the $\phi_j$ can be written as linear combinations of solutions to the linearized problem.  Finally, for those lattice sites at the front interface, we must match the linearly decaying front ahead of the front interface with those sites at capacity behind the front interface.  Inspecting the form of the front, we see that one condition is generated at each generation for which the front does not advance.  In the example above, this occurs at the first, second, and fourth generations at the first lattice site below capacity.  

This exercise motivates the remainder of the paper as follows.  First, we will study exponentially decaying solutions of the linearized equation and isolate $q-p$ such solutions from which to construct the front.  Then, matching conditions will be derived at the $q-p$ generations at which the front does not advance.  These conditions will be solved to yield formulas for the traveling front solution.  Finally, bounds on the locking region in parameter space are obtained by verifying that the post-migration population density remains above or below the critical population density $c$ at each generation.  

In the process of deriving the front solution, several questions arise that we will address.  For one, it will turn out that most of the linear solutions which form the building blocks of the front will be oscillatory in space.  For the front to be relevant to the model described in (\ref{eq:main}), it must be positive.  We will verify that the linear combination of these (mostly) oscillatory terms is, in fact, positive.  Second, there is also some question as to which $q-p$ linearly decaying solutions to include in the front construction.  Based upon the PDE theory, we will initially proceed by using the $q-p$ with the smallest modulus.  This choice will be substantiated by a spectral analysis of the problem where we will show that the inclusion of any other weaker decaying terms would lead to less desirable stability properties for the front.

\section{Properties of the linearized system}\label{sec:prelim}
In this section, we study of the dynamics for the linearization near the unstable zero state.  The linearized equation is  described by  
\be u_{i,t+1}=r\left( \frac{m}{2}u_{i-1,t}+(1-m)u_{i,t}+\frac{m}{2}u_{i+1,t}\right). \label{eq:linear} \ee
We seek exponentially decaying solutions of the form
\be u_{i,t}=\lambda^t \gamma^i, \label{eq:linearform} \ee
where $\gamma$ is the decay rate in space and $\lambda$ is the associated growth factor. We introduce the shorthand notation
\[ a=\frac{rm}{2}, \quad b=r(1-m), \]
and after plugging (\ref{eq:linearform}) into (\ref{eq:linear}), we obtain the dispersion relation
\[ \lambda(\gamma)=\frac{1}{\gamma}\left(a+b\gamma+a\gamma^2\right), \]
which relates the exponential decay in space of the solution to its temporal growth rate.  The speed associated to each decay rate $\gamma\in\mathbb{R}$ is called its envelope velocity $s_{\mathrm{env}}(\gamma)$ and can be calculated by solving $u_{i,t+q}=u_{i-p,t}$ using (\ref{eq:linearform}), from which we obtain 
\be s_{\mathrm{env}}(\gamma)=-\frac{\log(\lambda(\gamma))}{\log(\gamma)}. \label{eq:senv}\ee
Suppose that we began with initial data for (\ref{eq:linear}) that was localized in space.  Then a comparison argument shows that the  spreading speed of this solution (recall we are dealing with the linearized equation (\ref{eq:linear})) must be less than $s_{\mathrm{env}}(\gamma)$ for any $0<\gamma<1$.  We therefore define the {\em linear spreading speed} as 
\[ s_{\mathrm{lin}}=\min_{0<\gamma<1} s_{\mathrm{env}}(\gamma).\]
Associated to this speed is the  {\em linear decay rate} $\gamma_{\mathrm{lin}}$ which satisfies
\[ s_{\mathrm{env}}(\gamma_{\mathrm{lin}})=s_{\mathrm{lin}} \] 

\begin{figure}[!t]
\centering
\includegraphics[width=0.45\textwidth]{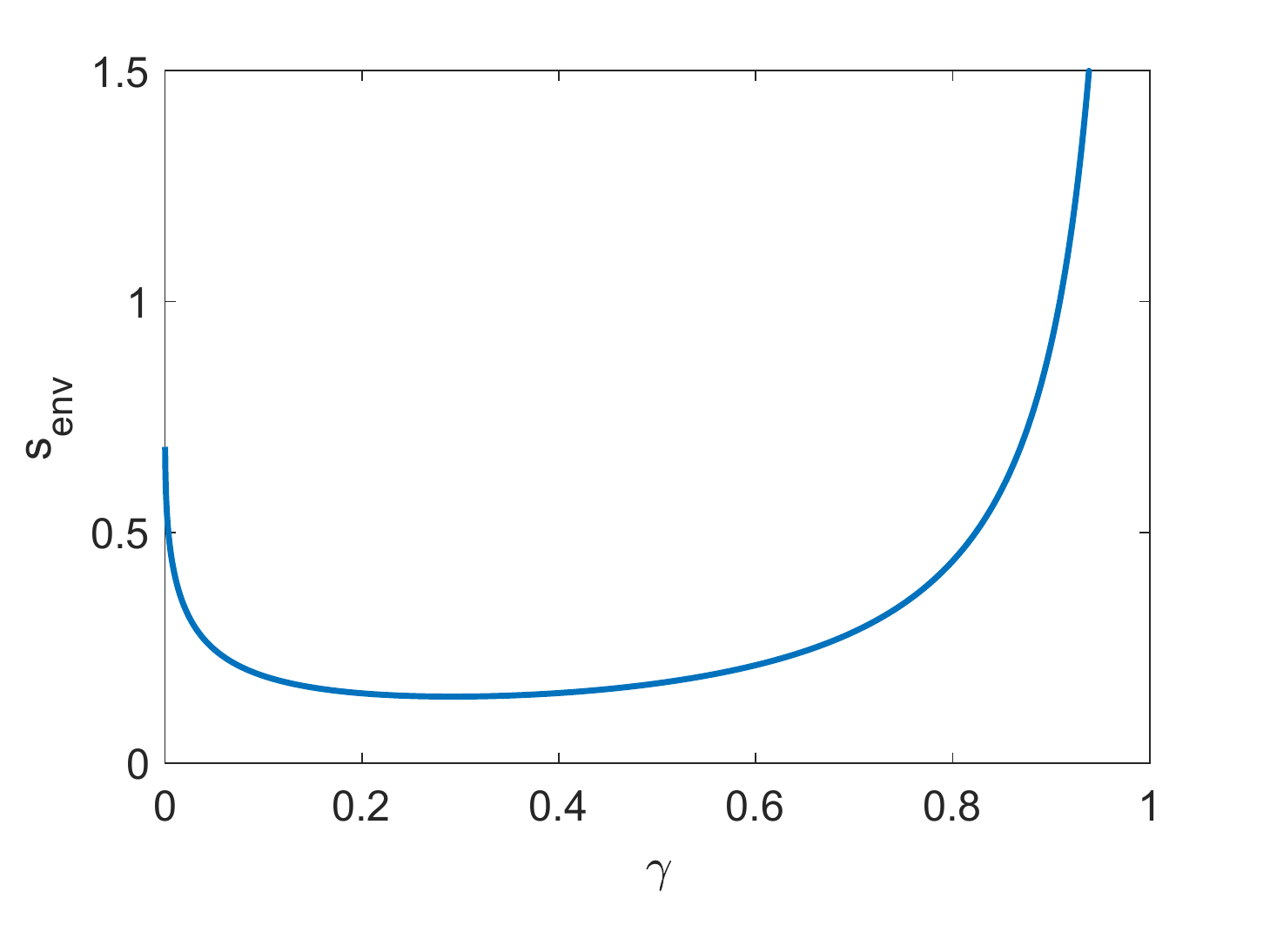}
\includegraphics[width=0.45\textwidth]{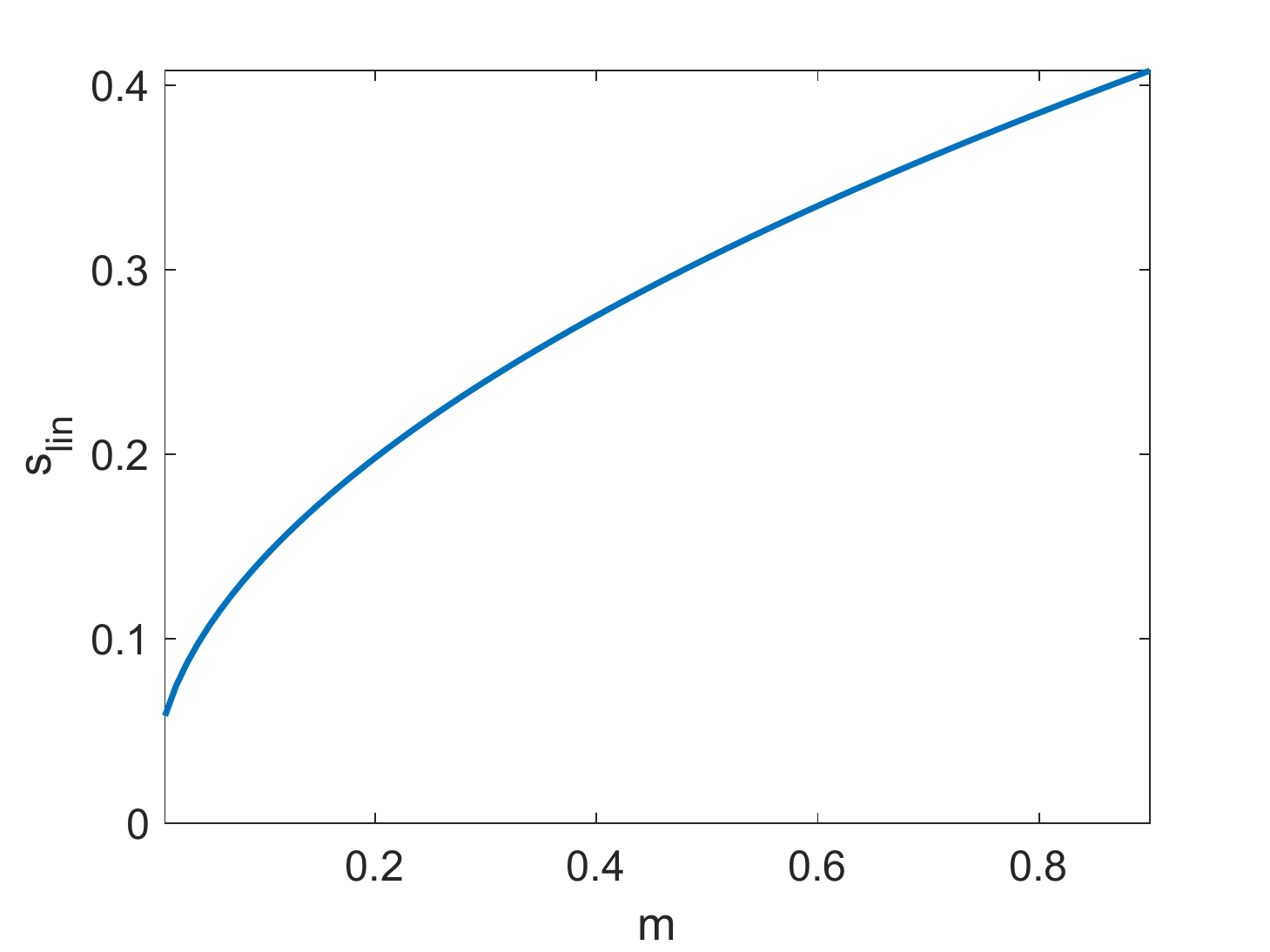}
\caption{On the left is the envelope speed $s_{env}$ as a function of the decay rate $\gamma$ for the parameter values $r=1.1$ and $m=0.1$.  The minimum value corresponds to the linear spreading speed, which for these parameter values is approximately $0.1443$. On the right is the linear spreading speed for $r=1.1$ and varying values of $m$.   }
\label{fig:senv}
\end{figure}
We will collect some facts regarding $s_{\mathrm{env}}(\gamma)$ and $s_{\mathrm{lin}}$.  

\begin{lemma}\label{lem:slin} If $1<r<\frac{2}{m}$, then $s_{\mathrm{env}}(\gamma)$ has a unique minimum, and $\slin$ is well defined with $\slin<1$.  Moreover, for any $1>\frac{p}{q}>s_{lin}$ there exist exactly two decay rates $0<\gamma_s<\gamma_w<1$ such that $s_{\mathrm{env}}(\gamma_s)=s_{\mathrm{env}}(\gamma_w)=\frac{p}{q}$. \end{lemma}

\begin{proof} Note that $r<\frac{2}{m}$ is equivalent to $a<1$.  Express $ s_{\mathrm{env}}(\gamma)$ as
\[ s_{\mathrm{env}}(\gamma)=1-\frac{\log\left(a+b\gamma+a\gamma^2\right)}{\log{\gamma}}, \]
from which it is clear that $\lim_{\gamma\to 0}  s_{\mathrm{env}}(\gamma)=1$.  Apply the derivative:
\[  s_{\mathrm{env}}'(\gamma)=\frac{b+2a\gamma}{a+b\gamma+a\gamma^2}\frac{-1}{\log{\gamma}}+\frac{\log\left(a+b\gamma+a\gamma^2\right)}{\gamma \log^2 \gamma}. \]
Critical points therefore occur whenever
\[ \left(b\gamma+2a\gamma^2\right) (-\log \gamma) =- \left(a+b\gamma+a\gamma^2\right) \log\left(a+b\gamma+a\gamma^2\right). \]
Let
\[ F_1(\gamma)=  \left(b\gamma+2a\gamma^2\right) (-\log \gamma), \quad F_2(\gamma)=\left(a+b\gamma+a\gamma^2\right) (-\log\left(a+b\gamma+a\gamma^2\right)), \]
and note $\lim_{\gamma\to 0} F_1(\gamma)=0$, $F_1(1)=0$, $F_2(0)=-a \log a$, and $F_2(1)=-r\log(r)$.  Since $a<1$, we have $F_1(0)=0< F_2(0)$, while since $r>1$, we have $F_2(1)<0=F_1(1)$. Since these functions are continuous, there must be an intermediate value at which they are equal.  This gives the existence of a decay rate $\gamma$ such that $s_{\mathrm{env}}'(\gamma)=0$.  To show that this value is unique, we compute derivatives
\begin{eqnarray*} F_1'(\gamma)&=& (b+4a\gamma) (-\log \gamma) - (b+2a\gamma), \\
F_2'(\gamma)&=& (b+2a\gamma)(- \log\left(a+b\gamma+a\gamma^2\right))-(b+2a\gamma) .
\end{eqnarray*}
 We then see that if $\gamma<a+b\gamma+a\gamma^2$, then we have that $F_1'(\gamma)>F_2'(\gamma)$ for all $0<\gamma<1$, and therefore the intersection (and therefore the root of $s_{\mathrm{env}}'(\gamma)$) must be unique.  Define the quadratic function $p(\gamma)=a+(b-1)\gamma+a\gamma^2$, and note if $b>1$, then all coefficients are positive, and so $p(\gamma)>0$ for all $0<\gamma<1$.  If $b<1$, then note that $p(0)=a>0$, $p'(0)=(b-1)<0$, $p(1)=r-1>0$, and $p'(1)=r>1$, and the minimum of $p(\gamma)$ occurs at $(1-b)/(2a)$.  Computing the value at the minimum, we obtain
\[ a-\frac{(b-1)^2}{4a^2}=\frac{4a^2-(b-1)^2}{4a^2} =\frac{(r-1)\left(2rm-r+1\right)}{r^2m^2}>0,\]
where the last bound holds since $2rm-r+1=1-b+rm>0$.  The final part of the Lemma now follows from uniqueness of the zero of $s_{\mathrm{env}}'(\gamma)$.  
\end{proof} 

\begin{rmk} The restriction $a=\frac{rm}{2}<1$ is natural in the sense that a speed one front always exists in the case $a>1$, regardless of the value of $c$.  The front in this case is identically one to the left of the interface and identically zero to the right of the interface.  Therefore, the natural decay rate in this case is $\gamma=0$, which minimizes $s_{\mathrm{env}}(\gamma)$ on the interval $[0,1]$.  

A related point is that when $a>1$, it holds that $s_{\mathrm{lin}}>1$, with $\gamma_{\mathrm{lin}}<0$ corresponding to an oscillating front.  Of course, due to the nature of the model, the fastest possible invasion speed is one, and these faster fronts are therefore not observed.  This phenomena has previously been observed in \cite{browne} in the context of feed-forward networks where the fronts are referred to as frustrated.
\end{rmk}

\begin{lemma}\label{lem:mvar} Suppose that $1<r<\frac{2}{m}$.  Then for $0<m<1$, it holds that 
\[ \frac{d\slin}{dm}>0. \]
\end{lemma}
\begin{proof}
Define $\slin$ as $s_{\mathrm{env}}(\gamma)$ for the unique $\gamma$ such that $s_{\mathrm{env}}'(\gamma)=0$.  Then implicit differentiation gives
\[ \frac{d\slin}{dm}=\frac{\partial s_{\mathrm{env}}}{\partial \gamma}\frac{\partial \gamma}{\partial m} +\frac{\partial s_{\mathrm{env}}}{\partial \lambda}\frac{\partial \lambda}{\partial m}. \]
The first term is zero, and we calculate
\begin{eqnarray*} \frac{\partial s_{\mathrm{env}}}{\partial \lambda}\frac{\partial \lambda}{\partial m}&=& -\frac{1}{\lambda \log(\gamma)} \frac{1}{\gamma}\left( \frac{r}{2}-r\gamma +\frac{r}{2}\gamma^2\right) \\
&=& - \frac{1}{\log(\gamma)} \frac{ \left(\frac{1}{2}-\gamma+\frac{1}{2}\gamma^2 \right)}{m\left(\frac{1}{2}-\gamma+\frac{1}{2}\gamma^2 \right)  +\gamma } \\
&=& - \frac{1}{\log(\gamma)} \frac{ \left(\gamma-1\right)^2}{m\left(\gamma-1\right)^2  +2\gamma }>0.
\end{eqnarray*}
\end{proof}

Lemma~\ref{lem:slin} guarantees the existence of two decaying solutions to the linear problem (\ref{eq:linear}).  Recall from our discussion in Section~\ref{sec:outline} that we expect to require $q-p$ such solutions.  It will turn out that we will utilize $\gamma_s$ and $q-p-1$ other solutions.  We turn our attention to those solutions now.   Let $s=\frac{p}{q}$.  Then from the envelope velocity formula, we obtain
\[ \frac{p}{q}=-\frac{\log(\lambda(\gamma))}{\log(\gamma)}, \]
and unraveling this equation, we find that $\gamma$ must be a root of the polynomial
\be \gamma^{q-p}=\left(a+b\gamma+a\gamma^2\right)^q.\label{eq:mainpoly} \ee

\begin{lemma}\label{lem:relevantroots} Suppose that $1<r<\frac{2}{m}$ and consider  $1>s=\frac{p}{q}>s_{\mathrm{lin}}$.  Let $\gamma_s$ (strong decay) and $\gamma_w$ (weak decay) be the unique real values from Lemma~\ref{lem:slin} for which $s_{\mathrm{env}}(\gamma_{s,w})=\frac{p}{q}$ with $0<\gamma_s<\gamma_{\mathrm{lin}}<\gamma_w$.  Then there exist $q-p$ roots of (\ref{eq:mainpoly}) with modulus less than or equal to $\gamma_s$.  
\end{lemma}
\begin{proof} We will use Rouche's Theorem to count zeros of the polynomial $\gamma^{q-p}-\left(a+b\gamma+a\gamma^2\right)^q$.  Denote $f(\gamma)=\gamma^{q-p}$ which has a root of order  $q-p$  at the origin.  Denote $g(\gamma)=\left(a+b\gamma+a\gamma^2\right)^q$.  On the circle of radius $\gamma_s$, since $g(\gamma)$ is a polynomial with positive coefficients, we have that $g(\gamma_s)=f(\gamma_s)$ and $|g(\gamma)|< |f(\gamma)|$ for all other $|\gamma|=\gamma_s$.  Let $\epsilon>0$.  Since we are studying the minimal root $\gamma_s$, we see that $|f(\gamma)|$ is strictly larger than $|g(\gamma)|$ on the ball of radius $\gamma_s+\epsilon$ for $\epsilon$ sufficiently small.  Thus, Rouche's Theorem applies, and there are exactly $q-p$ roots inside this ball.  Since $\epsilon$ is arbitrary, the result holds as $\epsilon\to 0$ as well.  
\end{proof}

\begin{rmk} We have thus far considered fronts moving to the right with $s>0$.  Since (\ref{eq:main}) is invariant with respect to the change $i \ \mapsto -i$, our analysis would carry over to fronts propagating to the left with speed $s<0$.  To see this, consider one of the roots of (\ref{eq:mainpoly}) defined in Lemma~\ref{lem:relevantroots}.  Let $z=\frac{1}{\gamma}$. Then $z$ satisfies
\[ z^{p-q}=\left( a+\frac{b}{z}+\frac{a}{z^2}\right)^q, \]
which after rearranging can be expressed as 
\[ z^{p+q}=\left( a+bz+az^2\right)^q. \]
This is the same polynomial that is obtained if one sets $s=-p/q$ in (\ref{eq:senv}).  
\end{rmk}

\section{Locked Fronts}\label{sec:lockedfronts}
In this section, we construct locked fronts propagating at rational speed and obtain bounds on the regions in parameter space for which they exist.  Before treating the general case, we will demonstrate what these fronts look like in two specific cases.  We assume throughout the remainder of this paper that $r>1$ (giving instability of the zero state) and $rm<2$ (allowing for the existence of fronts with speed less than one).  

\subsection{Examples}
We present several examples.  Note that speed $\frac{1}{2}$ has been discussed elsewhere; see \cite{wang19}.  The next simplest case is speed $\frac{1}{3}$, which we discuss below.  We also consider the case of speed $\frac{2}{5}$ before generalizing to arbitrary rational speeds.  

{\bf Example} Speed $\frac{1}{3}$.   In this case the polynomial (\ref{eq:mainpoly}) has six roots. Whenever $\frac{1}{3}>s_{\mathrm{lin}}(r,m)$, there is a unique strong decay rate $\gamma_1$.  By Lemma~\ref{lem:relevantroots}, there are exactly two roots with modulus less than or equal to $\gamma_1$.  Label the second root $\gamma_2<0$ with $0<-\gamma_2<\gamma_1$.  We then assume that the front is given by a semi-infinite sequence of ones on the left, followed by a linear combination of the linear solutions $\gamma_j^i$ for each lattice site $i>0$ on the right.  That is, we seek a solution 
\[ \phi_i=\left\{ \begin{array}{cc} 1 & i\leq 0 \\ \sum_{j=1}^{p-q} k_j\gamma_j^i & i\geq 1 \end{array}\right. . \] 
Since the speed is $\frac{1}{3}$, we impose that three generations later, the front should have the same form but shifted to the right by one lattice site.   

Expanding the front over three generations, we will show below that the front evolves as follows: 
\[\begin{array}{ccccc}
\text{Lattice Site} & i=0 & i=1 & i=2 & i=3 \\
\text{Generation 0 } & 1 & \sum k_j \gamma_j & \sum k_j \gamma_j^2 & \sum k_j \gamma_j^3 \\
\text{Generation 1 } &1 & \sum k_j \gamma_j^{2/3} & \sum k_j \gamma_j^{5/3} & \sum k_j \gamma_j^{8/3} \\
\text{Generation 2 } &1 & \sum k_j \gamma_j^{1/3} & \sum k_j \gamma_j^{4/3} & \sum k_j \gamma_j^{7/3} \\
\text{Generation 3 } &1 & 1 & \sum k_j \gamma_j & \sum k_j \gamma_j^2
\end{array} \]
We must find conditions on the constants $k_j$ appearing in the linear combination that ensure that this is a solution, and we must verify the fractional powers appearing in intermediate generations.  

Rational roots of $\gamma_j$ are not uniquely defined, so we therefore use the first generation to define
\[ \gamma_j^{2/3}=\left(a+b\gamma_j+a\gamma_j^2\right),\]
and note for future reference that 
\[ \gamma_j^{-1/3}=\frac{1}{\gamma_j} \left(a+b\gamma_j+a\gamma_j^2\right). \]
Let us now justify the structure of the front stated above.  Recall that we say that a lattice site is at capacity if its population is one.  In each generation, if a lattice site has no parents at capacity, then the expression for the front at the that lattice site holds by virtue of the polynomial (\ref{eq:mainpoly}).  At all other lattice sites, conditions need to be imposed.  If the solution at a particular lattice site is below capacity but has a parent which is at capacity, then this enforces a condition on the constants $k_1$ and $k_2$. 


In this example, we see that conditions on the $k_i$ are enforced in generations one and two at the first lattice site below capacity.  In the first generation, we require
\[ \sum k_j\gamma_j^{2/3} = a+b\sum k_j \gamma_j+ a \sum k_j \gamma_j^2, \]
from which we note that if $k_1+k_2=1$, then this equation can be re-written as 
\[ \sum k_j\left(\gamma_j^{2/3} - a -b\gamma_j-a\gamma_j^2\right)=0,\]
and equality is seen to hold  by the definition of $\gamma_j^{2/3}$.  In the second generation, we instead require 
\be \sum k_j\gamma_j^{1/3} = a+b\sum k_j \gamma_j^{2/3}+ a \sum k_j \gamma_j^{4/3}, \label{eq:13gen2} \ee
and if 
\[ \frac{k_1}{\gamma_1^{1/3}}+\frac{k_2}{\gamma_2^{1/3}} =1, \]
then (\ref{eq:13gen2}) can be written as 
\[ \sum k_j\gamma_j^{-1/3}\left(\gamma_j^{2/3} - a -b\gamma_j-a\gamma_j^2\right)=0,\]
which is once again zero.  This determines a system of equations for $k_j$
\[ \left(\begin{array}{cc} 1& 1 \\ \gamma_1^{-1/3} & \gamma_2^{-1/3}   \end{array}\right) \left(\begin{array}{c} k_1 \\ k_2 \end{array}\right) = \left(\begin{array}{c} 1 \\ 1  \end{array}\right),\]
with solution
\[ \left(\begin{array}{c} k_1 \\ k_2 \end{array}\right) = \frac{1}{\gamma_2^{-1/3}-\gamma_1^{-1/3}} \left(\begin{array}{c} \gamma_2^{-1/3}-1 \\ 1-\gamma_1^{-1/3} \end{array}\right), \]
where the determinant can be simplified to 
\[ \gamma_2^{-1/3}-\gamma_1^{-1/3} = a\left(\gamma_2-\gamma_1\right)+a\left(\frac{1}{\gamma_2}-\frac{1}{\gamma_1}\right). \]
Note that the determinant is always negative in this case.  We argue geometrically that $k_1\gamma_1+k_2\gamma_2>0$.  The equations defining $k_1$ and $k_2$ can be interpreted as 
\[ \left(\begin{array}{c} k_1 \\ k_2\end{array}\right) \cdot \left(\begin{array}{c} 1 \\ 1\end{array}\right)  =1, \ \left(\begin{array}{c} k_1 \\ k_2\end{array}\right) \cdot \left(\begin{array}{c} \gamma_1^{-1/3} \\ \gamma_2^{-1/3}\end{array}\right) =1. \]
The ones vector is obviously in the first quadrant.  The vector $(\gamma_1^{-1/3} , \gamma_2^{-1/3})^T$ is in the fourth quadrant.  Moreover, since $-\gamma_2^{-1/3}>\gamma_1^{-1/3}$, we have that the angle between these two vectors exceeds $\frac{\pi}{2}$.  Therefore the angle $\theta=\tan^{-1}(k_2/k_1)$ must satisfy $-\frac{\pi}{4}<\theta<\frac{\pi}{4}$, and since $-\frac{\pi}{4}<\tan^{-1}(\gamma_2/\gamma_1)<0$, it follows that 
\[ \left(\begin{array}{c} k_1 \\ k_2\end{array}\right) \cdot \left(\begin{array}{c} \gamma_1 \\ \gamma_2\end{array}\right) =k_1\gamma_1+k_2\gamma_2 >0. \]
A similar argument works for the vector $(\gamma_1^i,\gamma_2^i)^T$ for all $i\geq 1$, and therefore we obtain positivity of the front.  Positivity of the front in all intermediate generations then follows since $ax+by+az>0$ if $x$, $y$, and $z$ are all positive. 

Finally, it remains to specify the values of $c$ which are compatible with the existence of the front.   In this example, one such condition is imposed in the second generation at the first lattice site below capacity.  The concern is that the population at this site will be so large so as to exceed the critical population density $c$ and thereby transition to  one following reproduction.  To avoid this, we require
\[ c>c_{\mathrm{min}}(r,m):= \frac{m}{2}+(1-m)\sum k_j\gamma_j^{2/3}+\frac{m}{2} \sum k_j \gamma_j^{5/3} . \]
A second condition is imposed in the second generation, where we require that sufficient population density occurs in the second position so that the reproduction function maps the population to capacity.  This requires
\[ c<c_{\mathrm{max}}(r,m):=  \frac{m}{2}+(1-m)\sum k_j\gamma_j^{1/3}+\frac{m}{2} \sum k_j \gamma_j^{4/3}. \]

{\bf Example} Speed $\frac{2}{5}$.  In this case, the polynomial (\ref{eq:mainpoly}) has ten roots, the smallest three of which are of interest to us.  Each of these three roots gives an exponentially decaying solution to the linearized equation (\ref{eq:linear}).  Once again, we seek a front solution given as a semi-infinite string of ones, followed by an exponentially decaying tail made up of a linear combination of the relevant roots.  
To solve for $k_j$, we expand the front over five generations: 
\[\begin{array}{ccccc}
\text{Generation 0 } & 1 & \sum k_j \gamma_j & \sum k_j \gamma_j^2 & \sum k_j \gamma_j^3 \\
\text{Generation 1 } &1 & \sum k_j \gamma_j^{3/5} & \sum k_j \gamma_j^{8/5} & \sum k_j \gamma_j^{13/5} \\
\text{Generation 2 } &1 & \sum k_j \gamma_j^{1/5} & \sum k_j \gamma_j^{6/5} & \sum k_j \gamma_j^{11/5} \\
\text{Generation 3 } &1 & 1 & \sum k_j \gamma_j^{4/5} & \sum k_j \gamma_j^{9/5} \\
\text{Generation 4 } &1 & 1 & \sum k_j \gamma_j^{2/5} & \sum k_j \gamma_j^{7/5} \\
\text{Generation 5 } &1 & 1 & 1 & \sum k_j \gamma_j 
\end{array}\]
Conditions on the constants $k_j$ are imposed in the first, second, and fourth generations.  In the first generation, we require
\[ \sum k_j \gamma_j^{3/5} = a+ b \sum k_j \gamma_j+ a\sum k_j \gamma_j^2. \]
Therefore if $\sum_j k_j=1$, we can substitute 
\[ \sum k_j \gamma_j^{3/5} = a\sum k_j+ b \sum k_j \gamma_j+ a\sum k_j \gamma_j^2, \]
and rearrange to find 
\[ 0= \sum k_j\left[ a+b\gamma_j+a\gamma_j^2 -\gamma_j^{3/5}\right], \] 
where equality holds since $\gamma_j$ is a root of (\ref{eq:mainpoly}).  Furthermore, we note that since there is some ambiguity in the definition of rational roots, this equation also serves to define the root
\be  \gamma_j^{3/5}=a+b\gamma_j+a\gamma_j^2. \label{eq:defof35} \ee
Since 3 and 5 are relatively prime, all other roots can be obtained by taking powers of $\gamma_j^{3/5}$ and $\gamma_j = \gamma_j^{5/5}$.

The second condition is imposed at the second generation, where we require 
\[ \sum k_j \gamma_j^{1/5} = a+ b \sum k_j \gamma_j^{3/5}+a \sum k_j \gamma_j^{8/5}. \]
In this case, if $\sum k_j \gamma_j^{-2/5} =1$, then we can substitute and use (\ref{eq:defof35}) to show equality.  The final equation to be satisfied occurs in the fourth generation and is 
\[ \sum k_j \gamma_j^{2/5} = a+ b \sum k_j \gamma_j^{4/5}+a \sum k_j \gamma_j^{9/5}, \]
and the condition $\sum k_j \gamma_j^{-1/5}=1$ implies that this condition is satisfied.   

We then have three equations for $k_j$ that take the form
\[ \left(\begin{array}{ccc} 1& 1 & 1 \\ \gamma_1^{-1/5} & \gamma_2^{-1/5} & \gamma_3^{-1/5} \\ 
\gamma_1^{-2/5} & \gamma_2^{-2/5} &  \gamma_3^{-2/5} \end{array}\right) \left(\begin{array}{c} k_1 \\ k_2 \\ k_3 \end{array}\right) = \left(\begin{array}{c} 1 \\ 1 \\ 1 \end{array}\right).\]
We recognize that the matrix is Vandermonde, and owing to the existence of explicit formulas for the determinant, we are able to solve the system using Cramer's rule as 
\[ \left(\begin{array}{c} k_1 \\ k_2 \\ k_3 \end{array}\right) =\left(\begin{array}{c} \frac{( \gamma_2^{-1/5}-1)( \gamma_3^{-1/5}-1)}{( \gamma_2^{-1/5}-\gamma_1^{-1/5})( \gamma_3^{-1/5}-\gamma_1^{-1/5})} \\  \frac{( \gamma_1^{-1/5}-1)( \gamma_3^{-1/5}-1)}{( \gamma_1^{-1/5}-\gamma_2^{-1/5})( \gamma_3^{-1/5}-\gamma_2^{-1/5})} \\  \frac{( \gamma_1^{-1/5}-1)( \gamma_2^{-1/5}-1)}{( \gamma_1^{-1/5}-\gamma_3^{-1/5})( \gamma_2^{-1/5}-\gamma_3^{-1/5})}\end{array}\right).   \]
Having determined the coefficients $k_j$, it remains to verify that the front solution is positive and to determine conditions on the critical population density $c$.    We return to the question of positivity later and leave the computation of critical $c$ values to the general case.

%

{\bf General Case} Speed $\frac{p}{q}$. We now consider $r>1$ and general rational speeds $s=\frac{p}{q}<1$ with $p$ and $q$ relatively prime.  Our main result is the following.  

\begin{theorem}\label{thm:main} Let $r>1$, and let $s=\frac{p}{q}<1$ with $p$ and $q$ relatively prime.  Then there exists a $m_*(r)$ and functions $c_{max}(r,m)$ and $c_{min}(r,m)$ such that for all $0<m<\mathrm{min}\{1,m_*(r)\}$ and all $c_{min}(r,m)<c<c_{max}(r,m)$,  there exists a positive traveling front solution to (\ref{eq:main}) with speed $s$.
\end{theorem}

The construction  mimics the examples worked out above.  Since $r>1$, the linear spreading speed is well defined.  By Lemma~\ref{lem:mvar}, we have that $s_{\mathrm{lin}}$ is monotone increasing in $m$.  Since $s_{\mathrm{lin}}\to0$ as $m\to 0$, we have that there exists a $m_*(r)\leq 1$ such that  $\frac{p}{q}>s_{\mathrm{lin}}$ for all $m<m_*(r)$.   Then for all $m<m_*(r)$, there exists exactly one real root $\gamma_1$ of (\ref{eq:mainpoly}) satisfying $0<\gamma_1<\gamma_{\mathrm{lin}}$.  By Lemma~\ref{lem:relevantroots}, there exist exactly $q-p$ roots with modulus less than or equal to $\gamma_1$, including the root $\gamma_1$.   Label these roots as $\gamma_j\in\mathbb{C}$.  For each $\gamma_j$, define the root
\be \gamma_j^{\frac{q-p}{q}}=\left(a+b\gamma_j+a\gamma_j^2\right). \label{eq:gammajrooteqn} \ee
Since $p$ and $q-p$ are relatively prime, the remaining roots can be obtained by taking powers of this one.  Now define the front  
\be \phi_i=\left\{ \begin{array}{cc} 1 & i\leq 0 \\ \sum k_j \gamma_j^i & i\geq 1 \end{array}\right..\label{eq:phi} \ee
Let $u_{i,0}=\phi_i$. Then using (\ref{eq:gammajrooteqn}), we calculate formally that 
\[  u_{i,t}=\mathrm{min}\left\{1, \sum k_j\gamma_j^{i-\frac{p}{q}t} \right\}, \]
provided that certain conditions on $c$ and $k_j$ are satisfied.  

Conditions on $k_j$ apply at each lattice site for which a parent lattice site is at capacity.  This occurs at each of the $q-p$ generations during which the front does not advance.  This leads to a system of linear equations that determine $k_j$.  Let 
\[ \zeta_j=\gamma_j^{-1/q}. \]
 The equations for $k_j$ lead to a solvability condition 
\be \left(\begin{array}{cccc} 1 & 1 & \dots & 1 \\ \zeta_1 & \zeta_2 & \dots &\zeta_{q-p} \\
\zeta_1^2 & \zeta_2^2 & \dots &\zeta_{q-p}^2 \\
\vdots & \vdots & \ddots & \vdots \\
\zeta_1^{q-p-1} & \zeta_2^{q-p-1} & \dots &\zeta_{q-p}^{q-p-1} \end{array}\right) \left(\begin{array}{c} k_1 \\ k_2 \\ \vdots \\ k_{q-p} \end{array}\right)= \left(\begin{array}{c} 1 \\ 1 \\ \vdots \\ 1 \end{array}\right).\label{eq:vander} \ee
Using Cramer's rule, the system can be solved explicitly, and we obtain
\be k_j=  \prod_{n\neq j} \frac{ \zeta_n-1}{ \zeta_n-\zeta_j}. \label{eq:kform} \ee

We now have shown that the sum in the definition of the traveling front (\ref{eq:phi}) is well defined.  Note that although the $\gamma_j$ and $k_j$ may be complex, the front is real.  This follows since the roots $\gamma_j$ appear in complex conjugate pairs.  This implies that the $\zeta_j$ also appear in complex conjugate pairs.  In turn, this implies that the column vectors of the matrix in (\ref{eq:vander}) appear in complex conjugate pairs.  Then since the linear combination of these columns prescribed by the $k_j$ is real it follows that the $k_j$ must also appear in complex conjugate pairs and the sum $\sum k_j\gamma_j$ is then a real number.

It remains to determine conditions on the critical population density parameter $c$ that are consistent with the existence of the front.  To do this, note that there are $p$ generations in which the front advances.  During each such generation, the population at that lattice site before reproduction must exceed the value of $c$.  This imposes the condition 
\[ c<\frac{m}{2}+(1-m)\sum k_j \gamma_j^{\tilde{p}/q} +\frac{m}{2}\sum k_j \gamma_j^{(\tilde{p}+q)/q} , \quad 1 \leq \tilde{p} \leq p. \]
The right hand side of this inequality is minimized for $\tilde{p}=p$ (we delay a proof of this fact until the following section), and we therefore define the upper boundary of allowable $c$ values as 
\be c_{\mathrm{max}}(r,m)=\frac{m}{2}+(1-m)\sum k_j \gamma_j^{p/q} +\frac{m}{2}\sum k_j \gamma_j^{(p+q)/q}. \label{eq:cmaxgen} \ee
On the other hand, during each of the $q-p$ generations for which the front does not advance, it is required that the population density is sufficiently small so that the solution does not transition to one.  This means we require
\[ \frac{m}{2}+(1-m)\sum k_j \gamma_j^{\tilde{p}/q} +\frac{m}{2}\sum k_j \gamma_j^{(\tilde{p}+q)/q}<c , \quad p+1 \leq \tilde{p} \leq q . \]
In this case, we present a proof in the following section that  the lower boundary of allowable $c$ values is
\be c_{\mathrm{min}}(r,m)=\frac{m}{2}+(1-m)\sum k_j \gamma_j^{(p+1)/q} +\frac{m}{2}\sum k_j \gamma_j^{(p+q+1)/q}. \label{eq:cmingen} \ee

It remains to validate that $c_{\mathrm{min}}(r,m)<c_{\mathrm{max}}(r,m)$ for all $m<m_*(r)$, as well as the fact that $\phi_i>0$.  We perform this analysis in the following section. 

\section{Front positivity and expansions of locking regions in the small migration limit}\label{sec:positive}

The purpose of this section is two-fold.  We will first restrict to small $m$ and show that  $c_{\mathrm{min}}(r,m)<c_{\mathrm{max}}(r,m)$.  Positivity of the front is also obtained in this process.  Subsequently, these facts will be extended to all $0<m<\mathrm{max}\{1, m^*(r)\}$ using a proof by contradiction.  In doing this, we will have completed the proof of Theorem~\ref{thm:main}.

 In section~\ref{sec:asy}, we derive asymptotic expansions for the roots $\gamma_j$, the terms $\zeta_j$, and the constants $k_j$.  The case of speed $\frac{1}{q}$ fronts is easist, so we begin with this analysis in section~\ref{sec:1q}.  We then extend our results to general rational speeds in section~\ref{sec:pq} and then extend to arbitrary values of $m$ in section~\ref{sec:largerm}.

\subsection{Asymptotic analysis in the small migration limit $m\to 0$}\label{sec:asy}
In this section, we consider the limit as the migration rate tends to zero ($m\to 0$) with the assumption that $r>1$ is held constant.  To leading order, this is equivalent to the limit $a\to 0$.  For most quantities of interest, the first order correction will also match, and so we proceed treating $a$ as a small parameter.  To begin, we require expansions for the $q-p$ roots $\gamma_j$. Let $N=q-p$. Then (\ref{eq:mainpoly}) reads
\[ \gamma^{N} =\left(a+b\gamma+a\gamma^2\right)^q. \]
To leading order, we therefore solve $\gamma^N=a^q$, and expanding further we are able to obtain 
\be \gamma_j= a^{\frac{q}{N}} \left(\omega_j+a^{\frac{p}{N}}\frac{bq}{N}\omega_j^2 +\text{h.o.t.} \right), \label{eq:gammajs} \ee
where $\omega_j$ are the $N$-th roots of unity, given by
\[ \omega_j=e^{\frac{2\pi(j-1)\mbi}{N}}. \]
We now consider $\zeta_j=\gamma_j^{-\frac{1}{q}}$.  To compute $\zeta_j$ and its expansion, we use the expression
\be \zeta_j=\frac{\gamma_j^{\ell_1}}{\left(a+b\gamma_j+a\gamma_j^2\right)^{\ell_2}}, \label{eq:zetaells} \ee
for some positive integers $\ell_1$ and $\ell_2$.  The constants must be chosen to satisfy the Diophantine equation $q\ell_1-N\ell_2=-1$.  Since $N$ and $q$ are relatively prime, we see that this equation has integer solutions.  Furthermore,  using Bezout's identity, we can also surmise  that $0<\ell_1<\ell_2\leq N$.    The following expansion for the $\zeta_j$ holds:
\be \zeta_j=a^{-\frac{1}{N}}\left( \omega_j^{\ell_1}-\frac{b}{N}a^{\frac{p}{N}}\omega_j^{\ell_1+1} +\text{h.o.t.} \right). \label{eq:zetajsina} \ee
Finally, using our expansions for $\zeta_j$, we find that $|\zeta_j-1|\leq \overline{C} a^{-\frac{1}{N}}$, while $|\zeta_j-\zeta_n|\geq \underline{C} a^{-\frac{1}{N}}$, so that 
\be |k_j|\leq C , \label{eq:kbd} \ee
for some $C$ independent of $a$.

{\bf Example} Speed $\frac{1}{3}$.  Recall that in this case $N=q-p=2$, and we will use the two roots of unity $\omega_1=1$ and $\omega_2=-1$.  Using (\ref{eq:gammajs}), we obtain expansions for the roots as follows:
\[ \gamma_1=a^{\frac{3}{2}}+\frac{3}{2}ba^2 +\text{h.o.t.}, \quad \gamma_2=-a^{\frac{3}{2}}+\frac{3}{2}ba^2 +\text{h.o.t.}.  \]
Since $q=3$ and $N=2$ we obtain $\ell_1=1$ while $\ell_2=2$ and using (\ref{eq:zetajsina})  we find
\[ \zeta_1=\frac{1}{\sqrt{a}}-\frac{b}{2} +\text{h.o.t.}, \quad \zeta_2=-\frac{1}{\sqrt{a}}-\frac{b}{2} +\text{h.o.t.}. \]
Next, 
\begin{eqnarray*}
k_1 &=& \frac{\zeta_2-1}{\zeta_2-\zeta_1} =\frac{ -\frac{1}{\sqrt{a}}-1-\frac{b}{2} +\text{h.o.t.}}{-\frac{2}{\sqrt{a}} +\text{h.o.t.}}  = \frac{1}{2}+\frac{1+\frac{b}{2}}{2}\sqrt{a}+\text{h.o.t.}, \\
k_2 &=& \frac{1-\zeta_1}{\zeta_2-\zeta_1} =\frac{ -\frac{1}{\sqrt{a}}+1+\frac{b}{2} +\text{h.o.t.}}{-\frac{2}{\sqrt{a}} +\text{h.o.t.}}= \frac{1}{2}-\frac{1+\frac{b}{2}}{2}\sqrt{a}+\text{h.o.t.}.
\end{eqnarray*}
We now obtain expansions for $c_{\mathrm{min}}(r,m)$ and $c_{\mathrm{max}}(r,m)$.  We use
\[ \gamma_1^{1/3}= \sqrt{a} +\frac{b}{2}a+\text{h.o.t.}, \quad \gamma_2^{1/3}= -\sqrt{a} +\frac{b}{2}a+\text{h.o.t.}, \]
so that 
\[ \sum k_j\gamma_j^{1/3} = (1+b)a+\text{h.o.t.}. \]
Recall  $c_{\mathrm{max}}(r,m)$ and write it in terms of $a$,
\[ c_{\mathrm{max}}(r,m)=  \frac{1}{r}\left(a+b\sum k_j\gamma_j^{1/3}+a \sum k_j \gamma_j^{4/3}\right).  \]
A naive inspection of the formulas for $k_j$ and $\gamma_j^{1/3}$ would suggest that that middle term should dominate, and we would expect a leading order expansion in terms of $\sqrt{a}$.  However, due to cancellation we instead find the expansion 
\[ c_{\mathrm{max}}(r,m)=\left(\frac{1}{2}+\frac{r}{2}+\frac{r^2}{2}\right)m +o(m) .\]
On the other hand, we have 
\[ \gamma_1^{2/3}=a+ba^{3/2}+\text{h.o.t.}, \quad  \gamma_2^{2/3}=a-ba^{3/2}+\text{h.o.t.},  \] 
and so we have the expansion
\[ c_{\mathrm{min}}(r,m) = \left(\frac{1}{2}+\frac{r}{2}\right)m+o(m). \] 
In particular, the width of the $\frac{1}{3}$ speed locking region is $\mathcal{O}(m)$ as $m\to 0$;  see Figure~\ref{fig:tongueasy}.

\begin{rmk} While we have already established positivity of the front in this case, we note that the leading order expansions of $k_j$ and $\gamma_j$ are insufficient to verify positivity of the front due to cancellation. This turns out to be true for general speeds $\frac{p}{q}$, and so we will need to adopt a different approach to show that the front is positive.
\end{rmk}

\subsection{Scalings of the locking region for the case $s=\frac{1}{q}$} \label{sec:1q}
It turns out that leading order scalings for $c_{\mathrm{min}}(r,m)$ and $c_{\mathrm{max}}(r,m)$ can be attained in a simpler fashion than the direct  method employed in the previous example.  We demonstrate how this works in the simplest case of $s=1/q$ and return to the general case in the next sub-section.  To simplify notation, let
\[\Gamma_n=\sum k_j\gamma_j^{n/q}. \]
Then $\phi_1=\Gamma_q$, and we can write the front solution over all $q$ generations in terms of $\Gamma_n$ (again for $s=\frac{1}{q}$) as
\[\begin{array}{ccccc}
\text{Lattice Site} & i=0 & i=1 & i=2 & i=3 \\
\text{Generation 0 } & 1 & \Gamma_{q} & \Gamma_{2q} & \Gamma_{3q} \\
\text{Generation 1 } &1 & \Gamma_{q-1} & \Gamma_{2q-1} & \Gamma_{3q-1} \\ 
\text{Generation 2 } &1 & \Gamma_{q-2} & \Gamma_{2q-2} & \Gamma_{3q-2} \\
\vdots &  \vdots  & \vdots & \vdots   &  \vdots  \\
\text{Generation $q-2$ } &1 & \Gamma_{2} & \Gamma_{2+q} & \Gamma_{2+2q} \\
\text{Generation $q-1$ } &1 & \Gamma_{1} & \Gamma_{1+q} & \Gamma_{1+2q} \\
\text{Generation $q$ } &1 & 1 & \Gamma_{q} & \Gamma_{2q}  \\
\end{array} \]
We can then re-express 
\[ c_{\mathrm{max}}(r,m)=\frac{m}{2}+(1-m)\Gamma_1 +\frac{m}{2}\Gamma_{1+q}, \]
and 
\[ c_{\mathrm{min}}(r,m)=\frac{m}{2}+(1-m)\Gamma_2 +\frac{m}{2}\Gamma_{2+q}. \]
Consulting (\ref{eq:gammajs}), we observe that 
\[ \Gamma_n=\O(m^{\frac{n}{N}}),\]
and consequently $\frac{m}{2}\Gamma_{1+q}$ and $\frac{m}{2}\Gamma_{2+q}$ are (at most) $\O(m^2)$. Therefore, to show that $c_{\mathrm{min}}(r,m)<c_{\mathrm{max}}(r,m)$ as $m\to 0$, we must show that $\Gamma_2<\Gamma_1$ in this limit.  Formally, this turns out be quite easy, as we note that these two quantities are related via 
\[ \Gamma_1=a+b\Gamma_2+a\Gamma_{2+q},\]
where we note that $\Gamma_{2+q}=\O(a)$ and $b>1$ so that $\Gamma_2<\Gamma_1$.  Of course, this relies on $\Gamma_2$ being positive, and so in order to make this argument rigorous (for small $m$), we must iterate this procedure to express $\Gamma_1$ and $\Gamma_2$ in terms of the quantities in the zeroth generation, where we recall that all $\Gamma_{jq}$ are $o(m)$.  We proceed as follows:
\begin{eqnarray*}
\Gamma_1&=& a+b\Gamma_2+a\Gamma_{2+q} \\
 &=& a+b\Gamma_2 +o(a) \\
&=& a+b(a+b\Gamma_3+a\Gamma_{3+q}) +o(a) \\
 &=& a+ba+b^2\Gamma_3 +o(a) \\
&\dots& \\
 &=& a+ba+b^2a+b^3a+\dots+b^{q-2}a +o(a).
\end{eqnarray*}
We have therefore obtained (after expressing $a$ and $b$ in terms of $r$ and $m$) that 
\[ \Gamma_1=\frac{m}{2} \sum_{j=1}^{q-1} r^j +o(m), \]
while 
\[ \Gamma_2=\frac{m}{2} \sum_{j=1}^{q-2} r^j +o(m),\]
from which we have $\Gamma_2<\Gamma_1$ and therefore $c_{\mathrm{min}}(r,m)<c_{\mathrm{max}}(r,m)$ for $m$ sufficiently small.   Repeating the argument above we can also show that $\Gamma_{n+1}<\Gamma_n$ for all $n$ therby validating our choice of $\Gamma_2$ in the formula for $c_{\mathrm{max}}(r,m)$.  Note also that a similar argument allows us to write the front $\phi_1=\Gamma_q=a\Gamma_1+o(a)$, and this implies positivity of the front itself.  

Finally, note that scalings for $c_{\mathrm{min}}(r,m)$ and $c_{\mathrm{max}}(r,m)$ are then obtained with leading order expansions
\[ c_{\mathrm{min}}(r,m)=\frac{m}{2} \sum_{j=0}^{q-2} r^j +o(m), \quad c_{\mathrm{max}}(r,m)=\frac{m}{2} \sum_{j=0}^{q-1} r^j +o(m). \]
A comparison between these expansions and the locking regions determined in Section~\ref{sec:lockedfronts} are shown in Figure~\ref{fig:tongueasy}.

\begin{figure}[!t]
\centering
\includegraphics[width=0.5\textwidth]{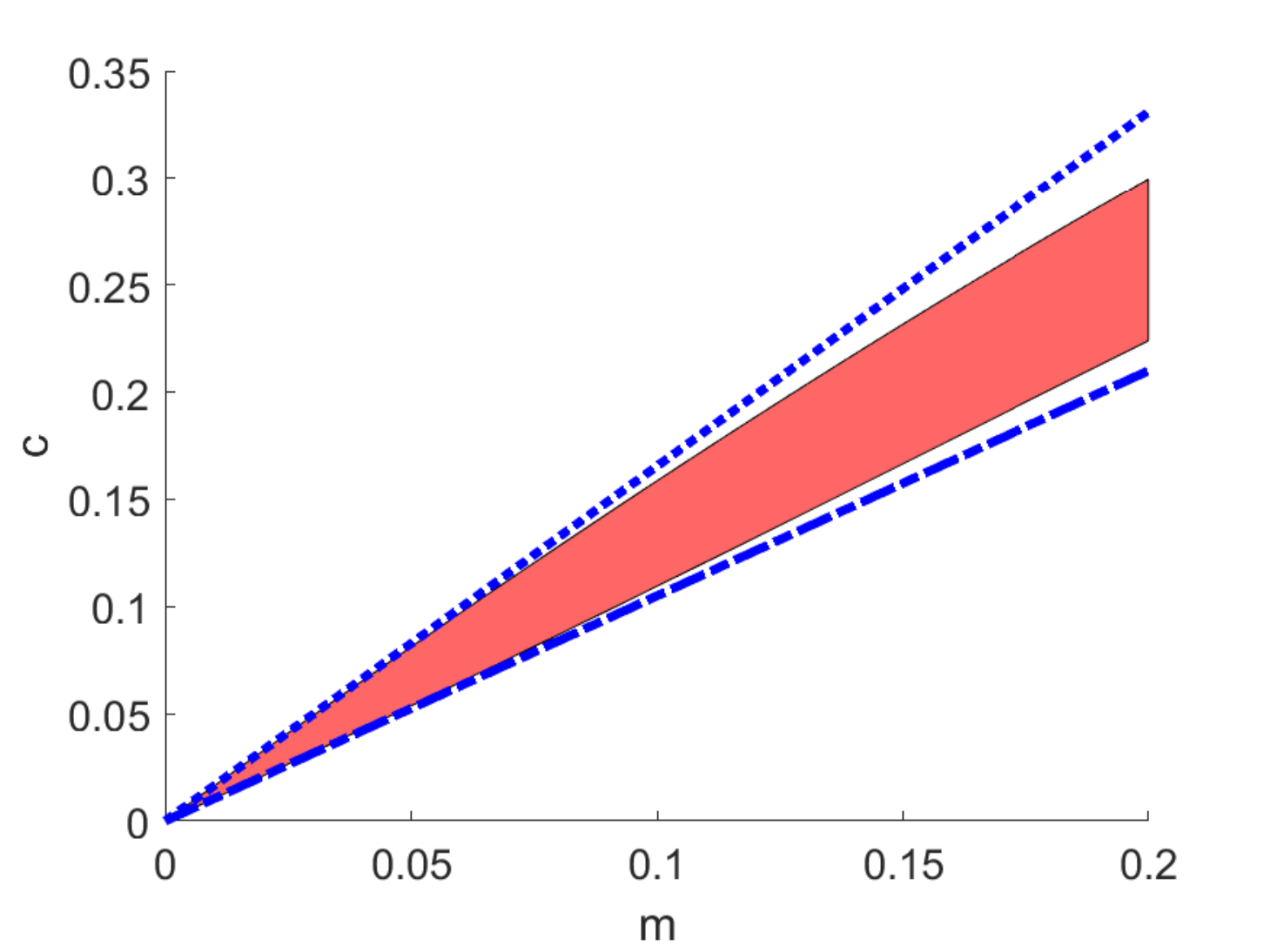}\includegraphics[width=0.5\textwidth]{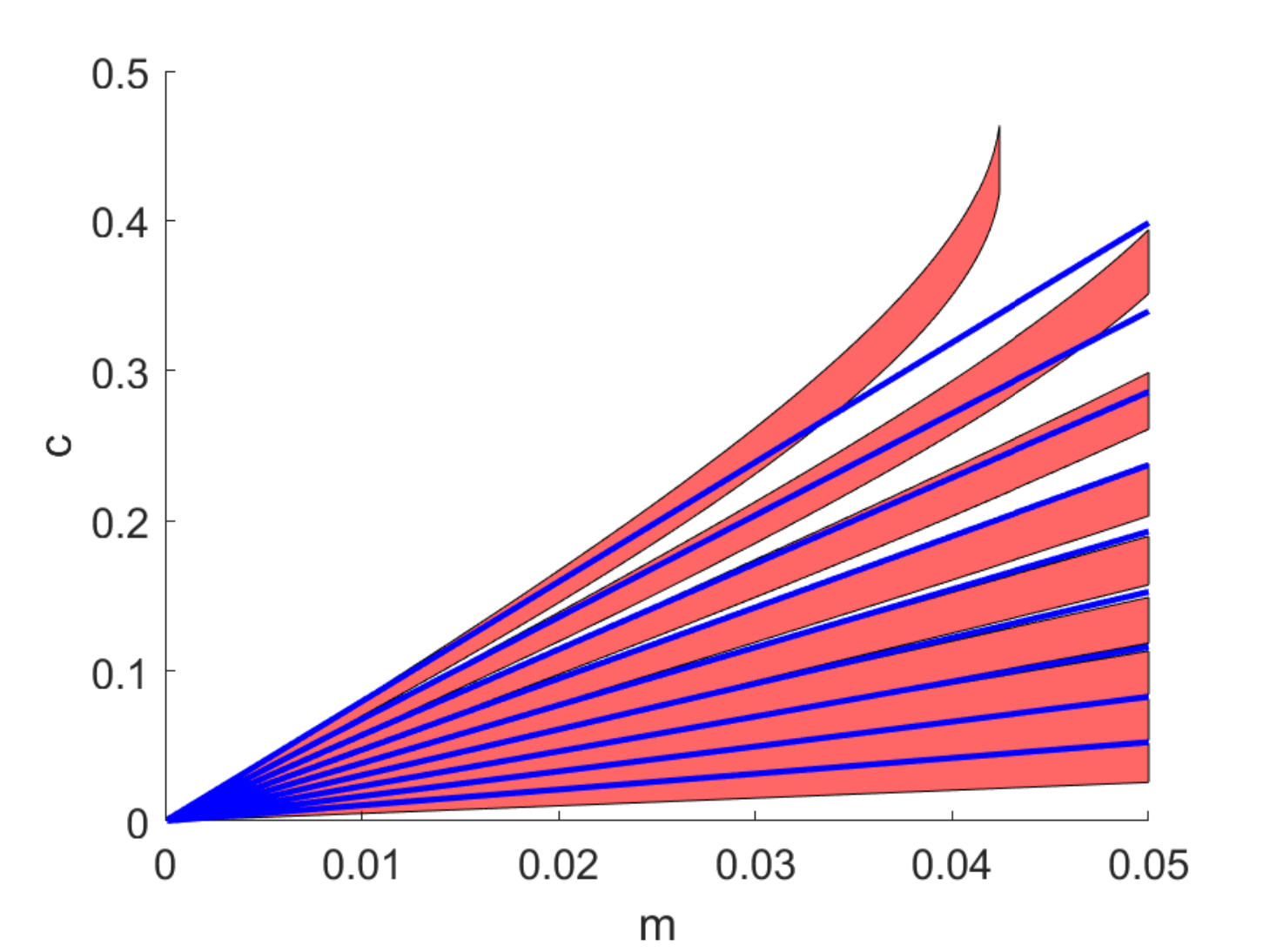}
\caption{Locking regions for speed one third (left) and speed $s=\frac{1}{q}$ for $q$ between two and ten (right).  The red shaded region are numerically computed using the formulas $c_{\mathrm{min}}(r,m)$ and  $c_{\mathrm{max}}(r,m)$; see (\ref{eq:cmingen}) and (\ref{eq:cmaxgen}).  The blue (dashed) lines depict leading order asymptotic expansions in the limit as $m\to 0$. }
\label{fig:tongueasy}
\end{figure}

\subsection{Scaling of locking regions for the general case $s=\frac{p}{q}$} \label{sec:pq}
We now consider the general case. We will obtain expansions for $\Gamma_p$ and $\Gamma_{p+1}$, establishing that $\Gamma_{p+1}<\Gamma_p$ for $m$ sufficiently small.  We start with $\Gamma_p$.  Once again, the goal is to express $\Gamma_p$ in terms of the population values in the previous generations and iterating the procedure to eventually obtain $\Gamma_p$ in terms of only $a$, $b$ and $\Gamma_{jq}$.  

For fixed $p$ and $q$, consider an integer $w$ with $1\leq w <q$ and then define the integers $z_i(w)\geq 0$ and $\eta_i(w)\in \{ 0, 1,2,\dots,p-1\}$ as follows:
\begin{eqnarray*}
q&=&z_1(w) p+w+\eta_1(w), \\
2q&=& z_2(w)p+z_1(w)p+w+\eta_2(w), \\
&\vdots& \\
pq &=& (z_p(w)+z_{p-1}(w)+\dots+z_1(w))p+w+\eta_p(w) .
\end{eqnarray*}
For fixed $w$, the $\eta_i$ are in fact a permutation of the integers $\{0, 1,2,\dots,p-1\}$.  Note also that if $w=p$,  then $\eta_p=0$ and  $(z_p+z_{p-1}+\dots+z_1+1)=q$.  
%
%

{\bf Example} Speed $\frac{3}{8}$.
We now work out an example that will illustrate the general argument.  For $s=\frac{3}{8}$, we will need to show that $\Gamma_4<\Gamma_3$.  In this example, we have
\be 8=3+3+2, \quad 16=3\cdot 3+3+3+1, \quad 24=3\cdot 3+3\cdot 3+3+3, \label{eq:z38} \ee
so that $z_1(3)=1$, $z_2(3)=3$, and $z_3(3)=3$, while $\eta_1(3)=2$, $\eta_2(3)=1$, and $\eta_3(3)=0$. In the left hand panel of the table below, we sketch the front solution over $q=8$ generations, ending with the one containing $\Gamma_3$.  Tracing the dependence on $a$ backwards through generations, we see that there is exactly $z_1=1$ lattice site at capacity located one lattice site to the left of $\Gamma_3$; there are $z_2=3$ lattice sites at capacity located two lattice sites to the left of $\Gamma_3$, and there are $z_3=3$ lattice sites at capacity located three lattice sites to the left of $\Gamma_3$.  In this way, we can recursively express $\Gamma_3$ in terms of its predecessors until $\Gamma_3$ is expressed as some function of $a$, $b$, and $\Gamma_{8j}$.  Since $\Gamma_{8j}=o(a^j)$, this estimate is sufficient to obtain an expansion for $\Gamma_3$ valid to $\O(a^3)$.   
\[\begin{array}{ccccc} 
1 & \Gamma_8 & * & * & \\
1 & \Gamma_5 &* & * & \\
1 & \Gamma_2 & * & * & \\
1 & 1 & \Gamma_{7} &* & \\
1 & 1 & \Gamma_{4} &* & \\
1 & 1& \Gamma_{1} & *& \\
1 & 1 &1 & \Gamma_6& \\
1 & 1 & 1 & \Gamma_3 & \\
\end{array} \quad\quad\quad \quad \quad 
\begin{array}{ccccc} 
\Gamma_1 & * & * & * & \\
1 & \Gamma_6 &* & * & \\
1 & \Gamma_3 & * & * & \\
1 & 1 & \Gamma_{8} &* & \\
1 & 1 & \Gamma_{5} &* & \\
1 & 1& \Gamma_{2} & *& \\
1 & 1 &1 & \Gamma_7& \\
1 & 1 & 1 & \Gamma_4 & \\
\end{array} 
 \] 

Now consider $\Gamma_4$.  Expanding as in (\ref{eq:z38}), we find
\[ 8=3+4+1, \quad 16=3\cdot 3+3+4+0, \quad 24=2\cdot 3+3\cdot 3+3+4+2. \]
Thus, $z_1(4)=1$ and $z_2(4)=3$, while $z_3(4)=2$.  The implication is that the predecessors of $\Gamma_4$ one lattice site to the left have $z_1(4)=1$ lattice site at capacity, while those two lattice sites to the left have $z_2(4)=3$ lattice sites at capacity.  The key difference is that $z_3(4)<z_3(3)$.  Therefore, to obtain a $\O(a^2)$ expansion for $\Gamma_3$ and $\Gamma_4$, we could work backwards $z_1+z_2=4$ generations and find that each could be written as a common function of $a$ and $b$, as well as a sum of other terms all dependent on $\Gamma_j$ with $j\geq 7$, and all $o(a^2)$.  Therefore, the expansions would agree to $\O(a^2)$.  To continue to $\O(a^3)$, however, we see that $\Gamma_4$ has one fewer at-capacity lattice sites to draw from, and its value will therefore be necessarily smaller.  We remark that a similar analysis could be performed to verify that $\Gamma_{n_1}\leq \Gamma_{n_2}$ for any $n_1>n_2$, thereby validating that our expressions for $c_{\mathrm{min}}$ and $c_{\mathrm{max}}$ are well-defined.

{\bf General Case} Speed $\frac{p}{q}$.
We will focus on $\Gamma_p$ and $\Gamma_{p+1}$ and show that $\Gamma_{p+1}<\Gamma_p$ for $m$ sufficiently small.  In the process, we will obtain that $\Gamma_p-\Gamma_{p+1}=\O(m^p)$ so that the width of the locking region is proportional to $m^p$.  

Define $z_i(p)$ and $\eta_i(p)$ as above and note that $\eta_p(p)=0$ and $1+z_1(p)+\dots+z_p(p)=q$.  Recall that $z_i(p)$ specifies the number of lattice sites a distance $i$ to the left of $\Gamma_p$ which are at capacity for previous generations of the front evolution.  Since all $\eta_i(p)>1$ for $i<p$, we then note that for any $j<p$,
\[ jq=z_j(p)p+\dots+z_1(p)p+p+1+\eta_i(p)-1, \]
so that $z_j(p+1)=z_j(p)$ and $\eta_i(p+1)=\eta_i(p)-1$ for any integer $1\leq j<p$.  For $j=p$, since $\eta_i(p)=0$, we must instead write
\[ pq=(z_p(p)-1)p+z_{p-1}(p)p+\dots +z_1(p)p +p+1 +p-1, \]
so that $z_p(p+1)=z_p(p)-1$.  

Therefore, tracking both $\Gamma_p$ and $\Gamma_{p+1}$ backwards $q$ generations, we find that there exists a function $\Phi_q(a,b,z_1,z_2,\dots,z_p)$ and linear maps $\Lambda_j$
such that 
\[ \Gamma_p=\Phi_q(a,b,z_1(p),\dots,z_p(p))+\Lambda_p\left(\Gamma_{q},\Gamma_{2q},\dots,\Gamma_{q(q-p)}\right). \] 
Furthermore, by (\ref{eq:gammajs}) we have that $\Gamma_{jq}=o(a^j)$, so that $\Lambda_p$ provides a contribution that is $o(a^p)$.  We also have that 
\[ \Gamma_{p+1}=\Phi_q(a,b,z_1(p+1),\dots,z_p(p+1))+\Lambda_{p+1}\left(\Gamma_{1},\Gamma_{1+q},\dots,\Gamma_{q(q-p+1)+1}\right). \]

Consider $\mathbf{z},\mathbf{w}\in (\mathbb{Z}^+)^p$  and let $\mathbf{z}\prec\mathbf{w}$ be the lexicographic ordering, where $\mathbf{z}\prec\mathbf{w}$ means $z_j<w_j$ for some $j$ while $z_k=w_k$ for all $1\leq k<j\leq p$.  Then for $a$ sufficiently small, $\Phi_p(a,b,\mathbf{z})$ is monotone increasing with respect to $\prec$.  Then since $\mathbf{z}(p+1)\prec \mathbf{z}(p)$, we have that $\Gamma_{p+1}<\Gamma_p$.  This argument can be generalized to show that $0<\Gamma_{n+1}< \Gamma_n$ for all $n\geq 1$.  This justifies our choice of of $\Gamma_j$ in the definitions of $c_{\mathrm{min}}(r,m)$ and $c_{\mathrm{max}}(r,m)$, proves that the interval $(c_{\mathrm{min}}(r,m),c_{\mathrm{max}}(r,m))$ is nonempty for fixed $r$ and small $m$ and guarantees positivity of the front (again for small $m$).  We also obtain that the width of the speed $s=\frac{p}{q}$ locking region scales with $\O(m^p)$.

\subsection{Extension to larger values of $m$}   \label{sec:largerm}

Let $m$ be sufficiently small and select parameters $r$ and $c$ so that the existence of a positive front with speed  $s=\frac{p}{q}$ is guaranteed. We now increase $m$ and show that positivity is preserved.  We argue by contradiction and assume that we can change parameters continuously so that we remain within the  speed $\frac{p}{q}$ locking region.  This is done until a set of parameters $(c,r,m)$ is reached at which  the front attains a zero value at one or more lattice sites.  Suppose for the moment that this occurs at a single lattice site.  Then one generation later, since the coefficients in (\ref{eq:main}) are positive, it must be the case that the value of the front at all lattice sites is positive.  This holds for all subsequent iterations, and so it is not possible for $q$ iterations of the (\ref{eq:main}) to return some lattice site to zero.  A similar argument works if more than one lattice site attains a zero value, even if the number of said lattice sites is not finite.  Finally, it is not possible for all lattice sites to attain zero simultaneously for a front with speed $s<1$.  This establishes positivity of the front for all parameters within the speed $\frac{p}{q}$ locking region. 

In a similar fashion, we can demonstrate that for fixed $r$ and any $0<m<\mathrm{min}(1,m^*(r))$ it holds that $c_{\mathrm{min}}(r,m)<c_{\mathrm{max}}(r,m)$.  This amounts to showing that $\Gamma_n>\Gamma_{n+1}$ for all $n$.  This holds for $m$ sufficiently small by the analysis in the previous subsection.  Now increase $m$ while keeping $c>c_{\mathrm{min}}(r,m)$.  The front solution will be well defined so long as $m$ remains below $\mathrm{min}(1,m^*(r))$ and $c_{\mathrm{max}}(r,m)>c_{\mathrm{min}}(r,m)$.   Suppose that in doing so $\Gamma_n=\Gamma_{n+1}$ for some value of $r$ and $m$.  Each of these quantities can then be expressed in terms of values of $\Gamma_j$ taken from the previous generation.  Since $a$ and $b$ are positive it holds that $\Gamma_n=\Gamma_{n+1}$ for the first time  if and only if $\Gamma_{n+p}=\Gamma_{n+p+1}$, $\Gamma_{n+p+q}=\Gamma_{n+p+q+1}$ and $\Gamma_{n+p-q}=\Gamma_{n+p-q+1}$ (if $n+p-q>0$).   If $n+p-q=0$ this yields a contradiction as we would then require $1=\Gamma_1$.  If $n+p-q\neq 0$ then we can continue this process to write $\Gamma_{n+p-q}$ and $\Gamma_{n+p-q+1}$ in terms of their predecessors until such a contradiction is obtained.

\section{Spectral Stability}\label{sec:stab}
In this section, we establish (strict)  spectral stability of the locked fronts constructed in previous sections.  Spectral stability (in weighted spaces) is a prerequisite for emergence of the front, and our analysis here will also substantiate our choice of  the $q-p$ steepest decaying terms $\gamma_j$ to include in the front construction.  

Consider a locked front with rational speed $s=\frac{p}{q}$. We follow \cite{chow95}; see also \cite{chow98,turzik08}.  Consider the Banach space $X=\ell^\infty(\mathbb{Z})$ with the supremum norm.   Let $G:X\to X$ be the generational map defined by (\ref{eq:main}).  Let $S:X\to X$ be the left shift operator defined by $(Su)_j=u_{j+1}$.  Locked fronts with speed $s=\frac{p}{q}$ are therefore fixed points of the map 
\[ \F(u)=S^{(p)} G^{(q)}(u). \]
We will linearize this map at the traveling front and study its spectrum.  We will fix ideas using a specific case and then generalize.   

{\bf Example} Speed $\frac{1}{2}$.  Let us begin with the simplest case of speed $s=\frac{1}{2}$.  Let $\phi$ be a locked front solution.  Since $N=1$, there is one relevant root of (\ref{eq:mainpoly}), and  we see that the front is described by the function
\[ \phi_i=\left\{ \begin{array}{cc} 1 & i\leq 0 \\ \gamma_1^i & i\geq 1 \end{array}\right. .\]
Next, we set $u=\phi+\eta$ and linearize $\F$ near the front.  For $i\leq 0$, due to the fact that $g'(1)=0$, we have that $(D\F(\phi)\eta)_i=0 $.  For any $i>1$, the linearization is the same as that of the constant state at zero, namely, 
\[ (D\F(\phi)\eta)_i = a^2\eta_{i-1}+2ab\eta_i+(b^2+2a^2)\eta_{i+1}+2ab\eta_{i+1}+a^2\eta_{i+2},  \]
while at the remaining value of $i=1$, we have 
\[ (D\F(\phi)\eta)_1 =2ab\eta_1+(b^2+2a^2)\eta_{2}+2ab\eta_{3}+a^2\eta_{4}. \]
Following \cite{chow95}, the spectrum of $D\F$ can be described in terms of its Fredholm properties and decomposed into continuous essential spectrum $\sigma_{\mathrm{ess}}(D\F)$ and point spectrum $\sigma_{\mathrm{pt}}(D\F)$, consisting of isolated eigenvalues of finite multiplicity.

The boundary of the essential spectrum is given in terms of two curves, which can be derived from the asymptotic operators near the homogeneous states zero and one.  Since the linearization near the stable state one is simply zero, this portion of the essential spectrum merely consists of the point at zero.  For the unstable zero state, we compute
\[ \partial\sigma_{\mathrm{ess}}(D\F)=\{ \lambda\in \mathbb{C} \ | \ \lambda= a^2e^{-\mbi k} +2ab+(b^2+2a^2)e^{\mbi k} +2abe^{2\mbi k}+a^2e^{3\mbi k}, \ k\in\mathbb{R}\}. \]
Since $a$ and $b$ are both positive, the most unstable portion of this curve occurs when $k=0$ and $\lambda=r^2$, reflecting the pointwise instability of the zero state with growth rate $r$ and the fact that $\F$ consists of the evolution over two generations.  It is important to note that this uniform growth is not observed if the perturbations are sufficiently localized in space.  We will employ exponential weights to control the decay of the perturbation and study the subsequent  impact on the spectrum.  To this end, suppose that the perturbation $\eta$ is localized so that $\sup_{i>0} \eta_i \bg^{-i} <\infty$ for some weight  $0<\bg<1$.  Consider the weighted space 
$X_\bg$ with norm $||u||_\bg=\sup u_i w_i$, where $w_i=\bg^{-i}$ for $i>0$ and $w_i=1$ otherwise.

 Then the boundary of the  essential spectrum associated to $D\F$ in the weighted space becomes
\[ \partial\sigma_{\mathrm{ess},\bg}(D\F)=\{ \lambda\in \mathbb{C} \ | \ \lambda= \frac{1}{\bg}a^2e^{-\mbi k} +2ab+(b^2+2a^2)\bg e^{\mbi k} +2ab\bg^2e^{2\mbi k}+a^2\bg^3 e^{3\mbi k}, \ k\in\mathbb{R}\}. \]
The most unstable point again occurs for $k=0$, where
\[ \lambda_{\mathrm{max}}= \frac{(a+b\bg+a\bg^2)^2}{\bg}. \]
Recall the values $\gamma_s=\gamma_1$ and $\gamma_w$ from Lemma~\ref{lem:relevantroots} that describe the strong and weak decay rates.   Also note that the right hand side of the previous equation is convex.  If we were to select the weight $\bg$ to be $\gamma_1=\gamma_s$, then we  would have that $\lambda_{\mathrm{max}}=1$, while for weight $\bg$ chosen as $\gamma_w$, we also have that $\lambda_{\mathrm{max}}=1$. Due to convexity, it follows that for any choice of weight between $\gamma_1=\gamma_s$ and $\gamma_w$, we have that the essential spectrum lies within the unit disk in the complex plane and is therefore stabilized. 

We now show that there is no unstable point spectrum.  To do so, we seek solutions to the eigenvalue equation $D\F(\phi)\eta=\lambda \eta$ for some $|\lambda|\geq 1$.   Since the linearization is zero for $i\leq 0$, we quickly obtain $\eta_i=0$ there.  For $i\geq 1$, we have 
\begin{eqnarray}
\lambda \eta_1&=&2ab\eta_1+(b^2+2a^2)\eta_{2}+2ab\eta_{3}+a^2\eta_{4},  \nonumber  \\
\lambda \eta_i&=& a^2\eta_{i-1}+2ab\eta_i+(b^2+2a^2)\eta_{i+1}+2ab\eta_{i+1}+a^2\eta_{i+2}, \quad i>1. \label{eq:sonehalfevalue} 
\end{eqnarray}
We will attempt to build eigenfunctions using a shooting method.  The first equation in (\ref{eq:sonehalfevalue}) can be solved for $\eta_4$, yielding a three dimensional shooting manifold.    
The second equation can be re-expressed as a difference equation satisfying 
\be \left(\begin{array}{c} \eta_{i+1} \\ \eta_{i+2} \\ \eta_{i+3} \\ \eta_{i+4} \end{array}\right) = \left(\begin{array}{cccc} 0 & 1 & 0 & 0 \\ 0 & 0 & 1 & 0 \\ 0 & 0 & 0 & 1 \\ -1 & -2\frac{b}{a} -\frac{\lambda}{a^2} & -\frac{b^2+2a^2}{a^2} & -2\frac{b}{a} \end{array}\right) \left(\begin{array}{c} \eta_{i} \\ \eta_{i+1} \\ \eta_{i+2} \\ \eta_{i+3} \end{array}\right).  \label{eq:diff} \ee
The characteristic polynomial for this dynamical system is 
\be (a+b\gamma+a\gamma^2)^2-\lambda \gamma=0. \label{eq:charpol} \ee
When $\lambda=1$, this polynomial is exactly (\ref{eq:mainpoly}), and there are four roots, with only $\gamma_1$ small enough so that the solution remains in $X_\gamma$.  For other values of $\lambda$ with $|\lambda|\geq 1$, the polynomial (\ref{eq:charpol}) can be rewritten as 
\[ \gamma=\frac{1}{\lambda} (a+b\gamma+a\gamma^2)^2, \]
and since the modulus of the right hand side is diminished when $|\lambda|\geq 1$, we can extend the argument using Rouche's Theorem from Lemma~\ref{lem:relevantroots} to show that there remains a unique root $\gamma_1(\lambda)$ with $|\gamma_1(\lambda)|\leq \gamma_1(1) $.  The eigenvector associated to this eigenvalue is, upon consulting (\ref{eq:diff}), given by $(1, \gamma_1(\lambda), \gamma_1^2(\lambda), \gamma_1^3(\lambda))^T$.   

To recap, we have shown that there is a three dimensional shooting manifold for which, if $\lambda$ is to be an eigenvalue,   must coincide with the one dimensional (strong) stable manifold of (\ref{eq:diff}).  However, since $\eta_0=0$, it turns out that we must have 
\[ \left(\begin{array}{c} 0 \\ \eta_1 \\ \eta_2 \\ \eta_3\end{array}\right) \in \mathrm{Span}\left\{  \left(\begin{array}{c} 1 \\ \gamma_1(\lambda) \\ \gamma_1^2(\lambda) \\ \gamma_1^3(\lambda) \end{array}\right)\right\}, \]
which is clearly not possible (aside from the trivial solution).  We have thus ruled out unstable (or marginally unstable) point spectrum.  In combination with our bounds on the essential spectrum in the weighted space $X_\bg$, we have therefore demonstrated strict spectral stability of the locked front propagating with speed one-half.

{\bf General Case} Speed $\frac{p}{q}$.
For general locked fronts of speed $\frac{p}{q}$, the method above can be adapted to once again yield stability.  We have the following result.  

\begin{theorem} \label{thm:stable} Fix $r>1$ and for $s=\frac{p}{q}<1$ let $\phi_i$ be the traveling front constructed in Theorem~\ref{thm:main}.  Then there exists a $0<\bg<1$ such that the front is spectrally stable in $X_\bg$.  
\end{theorem}

\begin{Proof} Recall that the map $\F$ in this case involves $q$ iterations of (\ref{eq:main}), followed by a shift of $p$ lattice sites to the left.  The boundary of the essential spectrum associated  to the unstable state has a point of maximal modulus when $k=0$ and $\lambda = r^q$.  In the weighted space $X_\bg$, this maximal point instead has real part 
\[ \lambda_{max}=\frac{(a+b\bg+a\bg^2)^q}{\bg^{q-p}}. \]
As was the case in the specific example considered above, the essential spectrum is stabilized for any weight $\gamma_s<\bg<\gamma_w$.

We now turn to the eigenvalue problem $D\F\eta=\lambda\eta$.  Assuming once again that the front interface is located at $i=0$, we see that  $\eta_i=0$ for all $i\leq 0$.  For $i>q+1$, we find
\be \lambda \eta_i=\sum_{j=-q}^{q} \alpha_{j+q} \eta_{p+i+j}, \label{eq:etai} \ee
where the $\alpha_j$ are the trinomial  coefficients of the polynomial $(a+b\gamma+a\gamma^2)^q$.  As in (\ref{eq:diff}), this recursion can be written as linear dynamical system in $2q$ dimensions.  There exists a (strong) stable eigenspace of dimension $N=q-p$ for the recursion, corresponding to those decaying solutions with rate greater than or equal to $\gamma_1$.   The equation for $\eta_1$ is 
\[ \lambda \eta_1=\sum_{j=-p}^q \alpha_{j+q}\eta_{p+1+j}, \]
which differs from (\ref{eq:etai}) in that the first $q-p$ terms are absent.  We will therefore seek $\eta_1$ through $\eta_{2q-N}$ such that
\[ \left(\begin{array}{c} 0  \\ \vdots \\ 0 \\ \eta_1 \\ \vdots \\ \eta_{2q-N} \end{array}\right) \in \mathrm{Span}\left\{ \left(\begin{array}{c} 1 \\ \gamma_1(\lambda) \\ \vdots  \\ \gamma_1^{N-1}(\lambda) \\ \vdots \\  \gamma_1^{2q-1} (\lambda) \end{array}\right), \dots \left(\begin{array}{c} 1 \\ \gamma_N(\lambda) \\ \vdots  \\ \gamma_N^{N-1}(\lambda) \\ \vdots \\  \gamma_N^{2q-1} (\lambda) \end{array}\right)\right\}. \]
Inspecting the first $N$ elements, we observe that a (non-trivial) inclusion is impossible, since the $N\times N$ Vandermonde matrix corresponding to the roots $\gamma_j(\lambda)$ has non-zero determinant.  We therefore obtain spectral stability of the linearization in the weighted space $X_\bg$.  

\end{Proof}

\section{Numerical Results}\label{sec:numerics}

In this section, we present numerical simulations of equation (\ref{eq:main}) and compare the observed invasion speeds to those predicted by the analysis of Section~\ref{sec:lockedfronts}.  

Direct numerical simulations of (\ref{eq:main}) were computed for a lattice consisting of $300$ to $400$ lattice sites.   Similar to \cite{wang19}, we use a domain shifting approach so that large number of generations may be simulated.  This approach works as follows: the first three lattice sites are initially set to capacity, while the remaining lattice sites are below capacity and rapidly converge to zero (we typically used zero initial conditions in these sites or some population density that decays faster than any exponential).  The system is then evolved using (\ref{eq:main}) until the fourth lattice site transitions to capacity.  At this point, the solution is then shifted to the left by one, and the site at the far right boundary is set to zero.  Speeds are then computed by calculating the number of shifts that occur and dividing by the total number of generations simulated.  Typically, an initial transient is discarded.  In the simulations presented in Figure~\ref{fig:obsvspredict}, the initial transient is $10,000$ generations, and the speed is calculated over the next $10,000$ generations.

The analysis in Section~\ref{sec:lockedfronts} reveals that the locking regions in parameter space are bounded by three curves.  We will again fix $r>1$ and vary the migration rate $m$ and the critical population density parameter $c$.   The rightmost point in the locking region is a vertical line at $m_*(r)$, where the linear spreading speed is the rational speed $p/q$.  For $m>m_*(r)$, there are no longer $q-p$ distinct roots near zero, and the construction in Section~\ref{sec:lockedfronts} no longer holds.  For $m<m_*(r)$, the boundaries in parameter space are given by the curves $c_{\mathrm{max}}(r,m)$ and $c_{\mathrm{min}}(r,m)$, which are given by formulas (\ref{eq:cmaxgen}) and (\ref{eq:cmingen}).  Numerical computation of these regions are presented in Figure~\ref{fig:tounges} as subsets of $(m,c)$ parameter space for two different choices of $r$.  We also present simulations that compare the observed invasion speed for different $m$ and $c$ values to those predicted by the analysis in Section~\ref{sec:lockedfronts}; see Figure~{\ref{fig:obsvspredict}.

\begin{figure}[!t]
\centering
\includegraphics[width=0.5\textwidth]{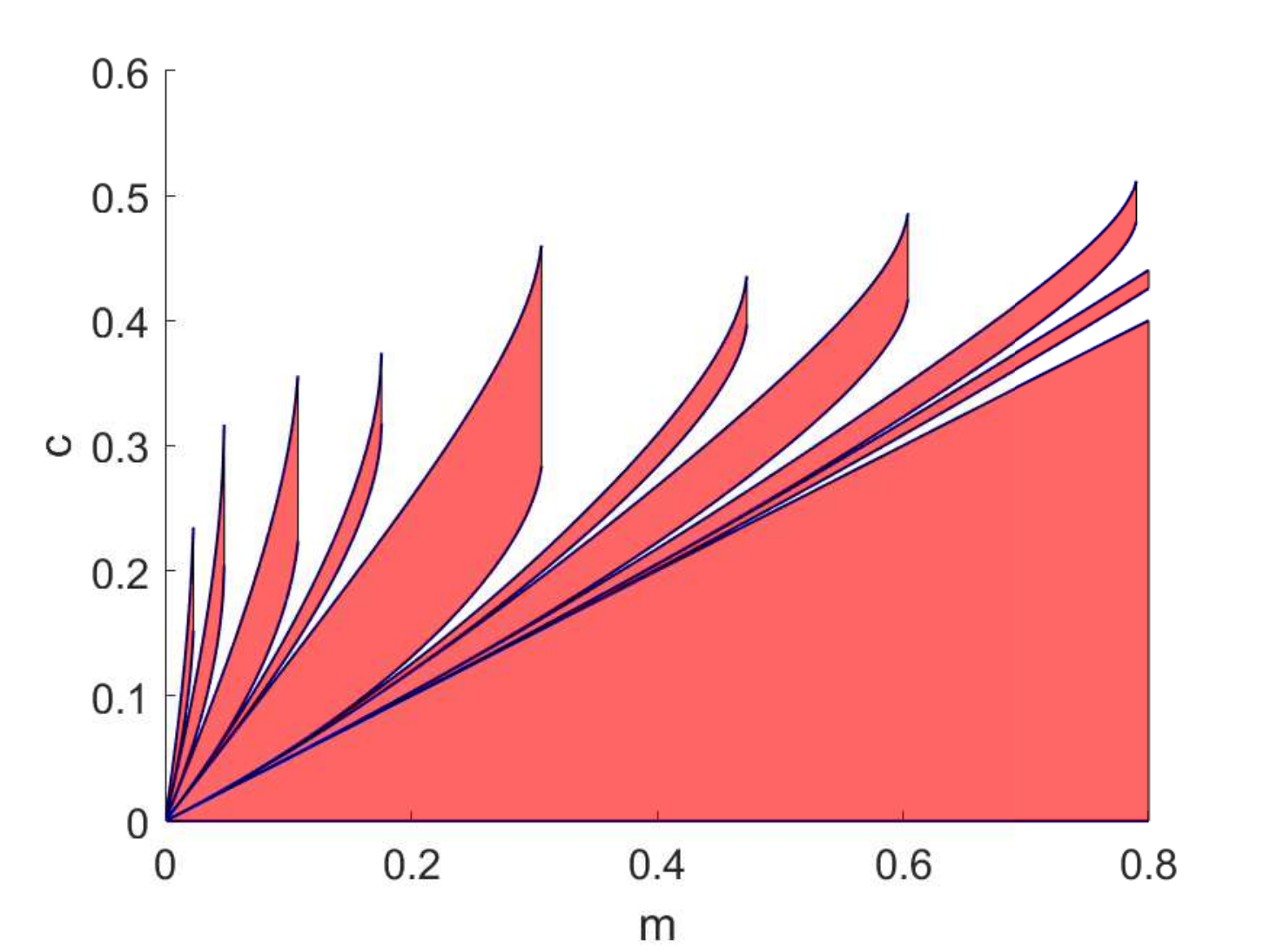}\includegraphics[width=0.5\textwidth]{lockedregionsNEWr1point1.pdf}

\caption{Locking regions (shaded) for all rational speeds $\frac{p}{q}$ with $q\leq 5$ and $1\leq p\leq q$ with $\mathrm{gcd}(p,q)=1$.  On the left is the case of $r=1.5$, while on the right is the case of $r=1.1$.  }
\label{fig:tounges}
\end{figure}

\begin{figure}[!t]
\centering
\includegraphics[width=0.5\textwidth]{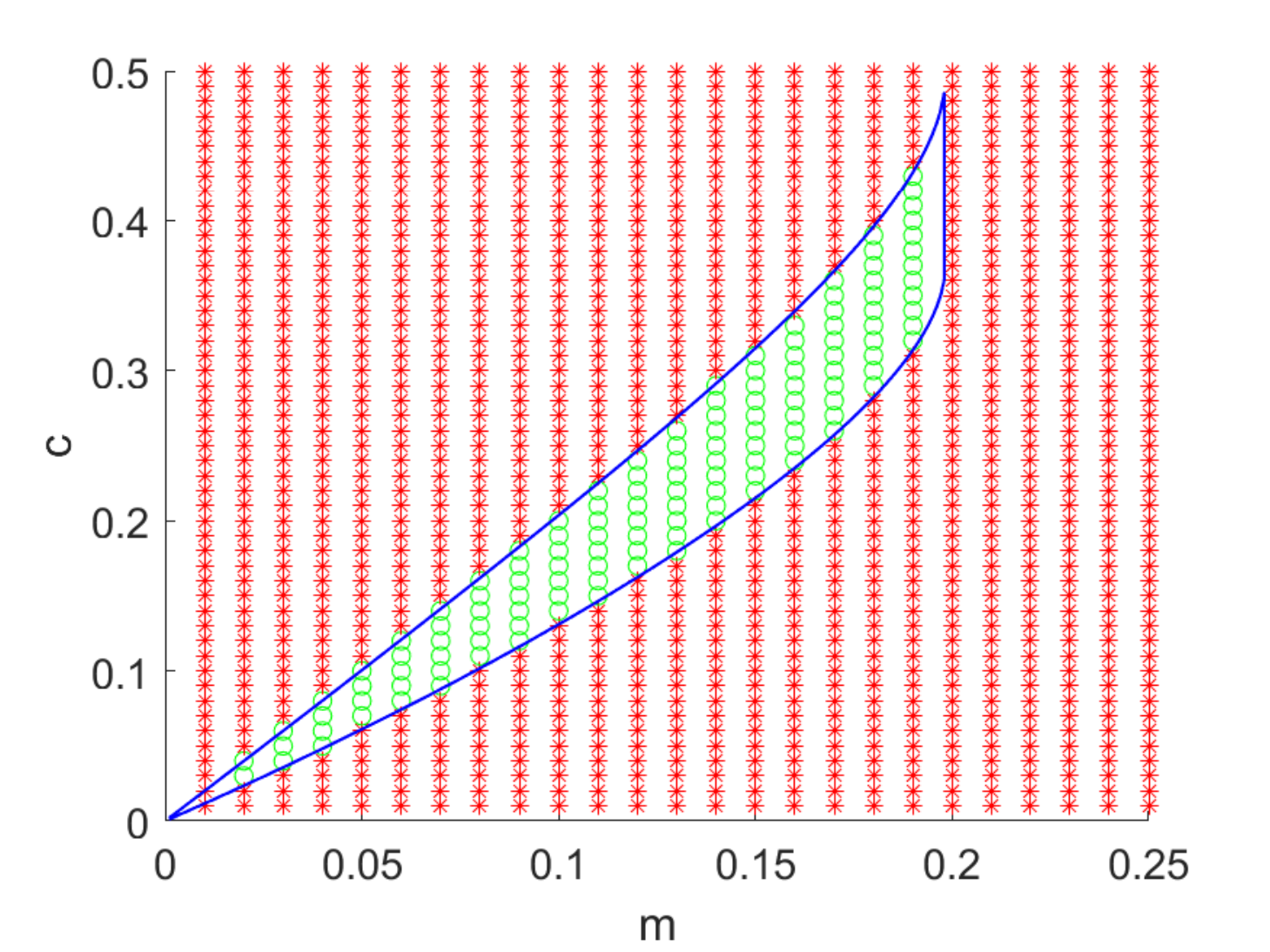}\includegraphics[width=0.5\textwidth]{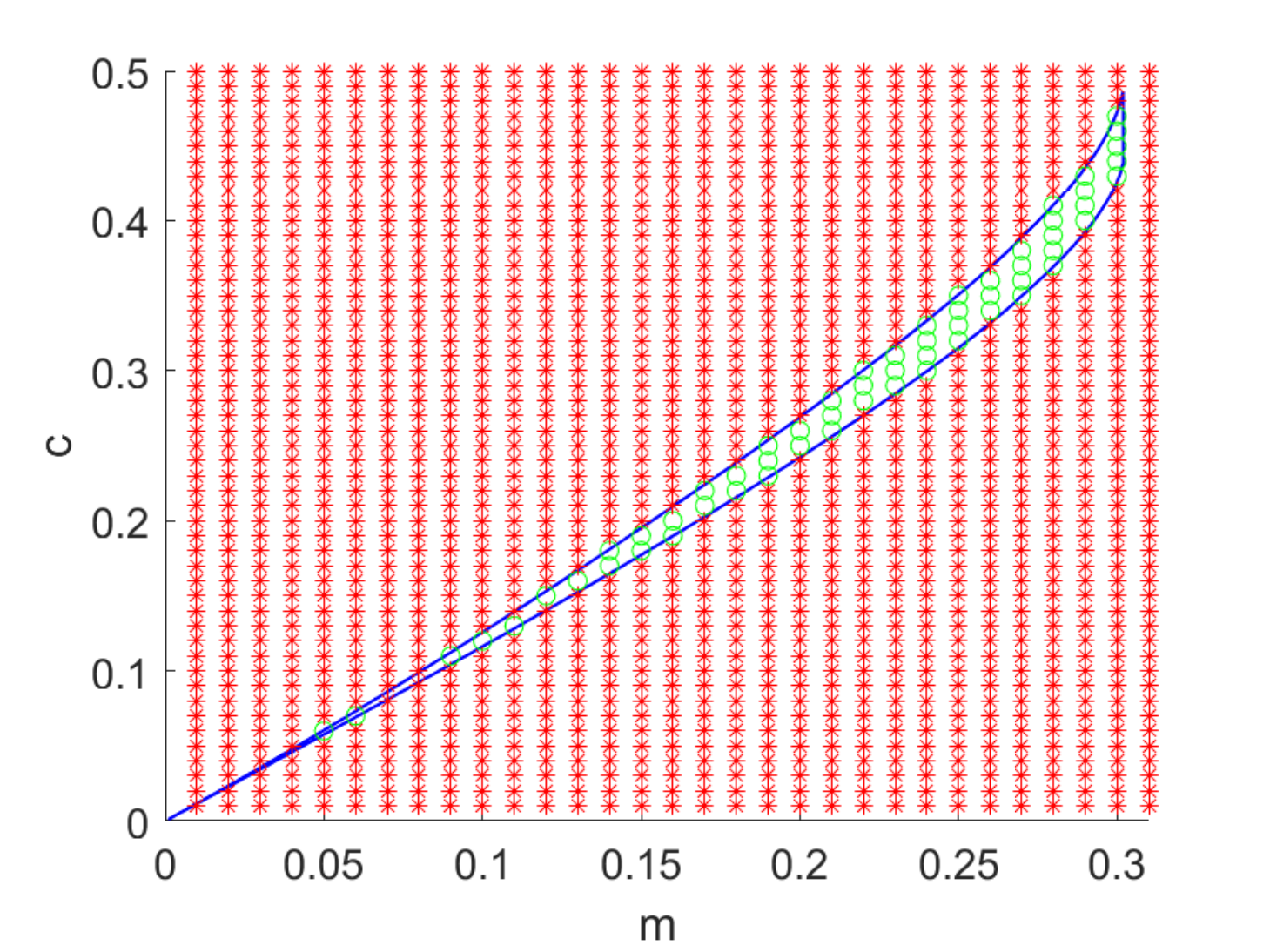}

\caption{Speed one third (left panel) and speed two fifths (right panel) velocity locking regions in $m-c$ parameter space with $r=1.3$.  Red asterisks show parameter values for which the numerically observed speed in direction simulations of (\ref{eq:main}) differs from the locked speed.  Green circles show those parameter values that lead to speed $\frac{1}{3}$ (left) or speed $\frac{2}{5}$ (right).  The blue curves depict the boundary of the locking regions derived from the construction of the traveling front in Section~\ref{sec:lockedfronts}.  }
\label{fig:obsvspredict}
\end{figure}

\section{Discussion}\label{sec:discussion}

The primary contribution of this paper was the construction of locked fronts for (\ref{eq:main}) for the piecewise linear reproduction function $g(u)$ in (\ref{eq:g}) and estimates for the boundary of their existence in parameter space.  We conclude with several directions for future research.

\paragraph{Pulled fronts and fronts with irrational speed.}  
Our construction of locked fronts with rational speeds uses the fact that locked fronts are fixed points of the map consisting of $q$ iterations of (\ref{eq:main}) followed by a shift of $p$ lattice sites.  One can imagine that this construction could be extended to pulled fronts propagating with (rational) linear spreading speeds.  One complication is that the root $\gamma_{\mathrm{lin}}$ is now a double root, so that the construction would involve $q-p+1$ roots $\gamma_j$ (counted with multiplicity).  The resulting solvability condition analogous to (\ref{eq:vander}) would then be underdetermined, and a family of fronts would exist.  The hope is that this flexibility could be utilized to satisfy the population density conditions that ensure that $c_{\mathrm{max}}(r,m)$ can be taken to be $\frac{1}{r}$.  Since this pulled front is a fixed point of a map, one might be tempted to expect locking to occur which is not consistent with observations from direct numerical simulations; see again Figure~\ref{fig:speedplots}.  In fact, we do expect this front to persist as $m$ is varied.  However, based upon our calculations in Section~\ref{sec:stab} and in analogy with the PDE theory, we anticipate a change in stability to occur as the migration rate is varied; see \cite{vansaarloos03} for a review of marginal stability.

Fronts with irrational speed are not fixed points of any map, so their construction would be more challenging still.  In the special case where $rc=1$ and the reproduction function is continuous, we would expect that a comparison principle argument could be used to prove the existence of pulled invasion waves; see for example \cite{weinberger82}.  Extensions to the case $rc<1$ are less clear.  

\paragraph{Scaling of locking regions.} For the locking regions studied here, the largest regions appear to be those with speed $1/q$; see Figure~\ref{fig:ltscaling}.  This is in contrast to the classical case of phase locking of rotation numbers for circle maps, where the largest measure locking regions are the ones corresponding  to  smaller $q$ values.  We also showed that locking regions for speed $s=p/q$ scaled with $\O(m^p)$.  It would be interesting to whether similar scalings hold for more general reproduction functions.  

\begin{figure}[!t]
\centering
\includegraphics[width=0.5\textwidth]{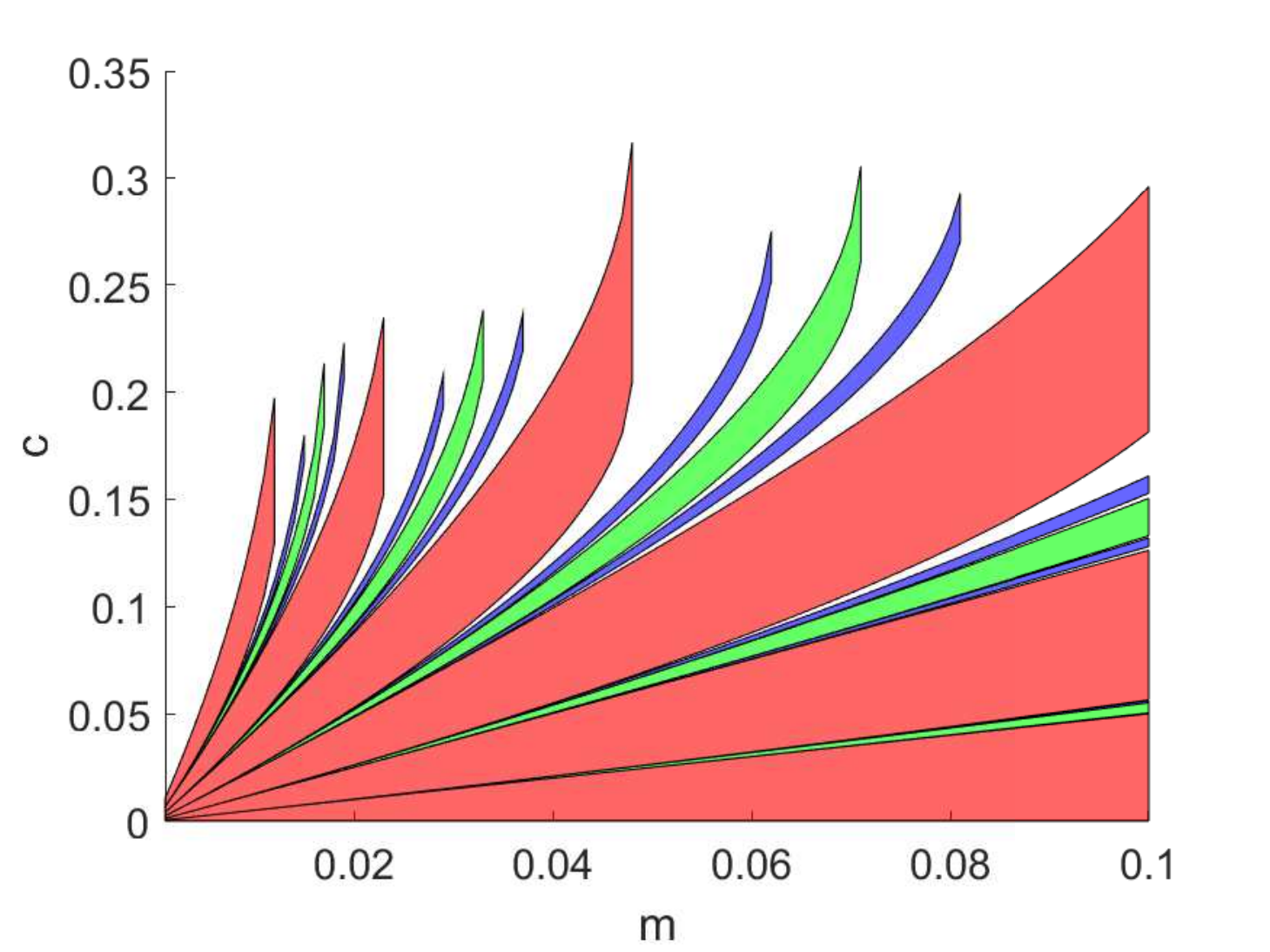}\includegraphics[width=0.5\textwidth]{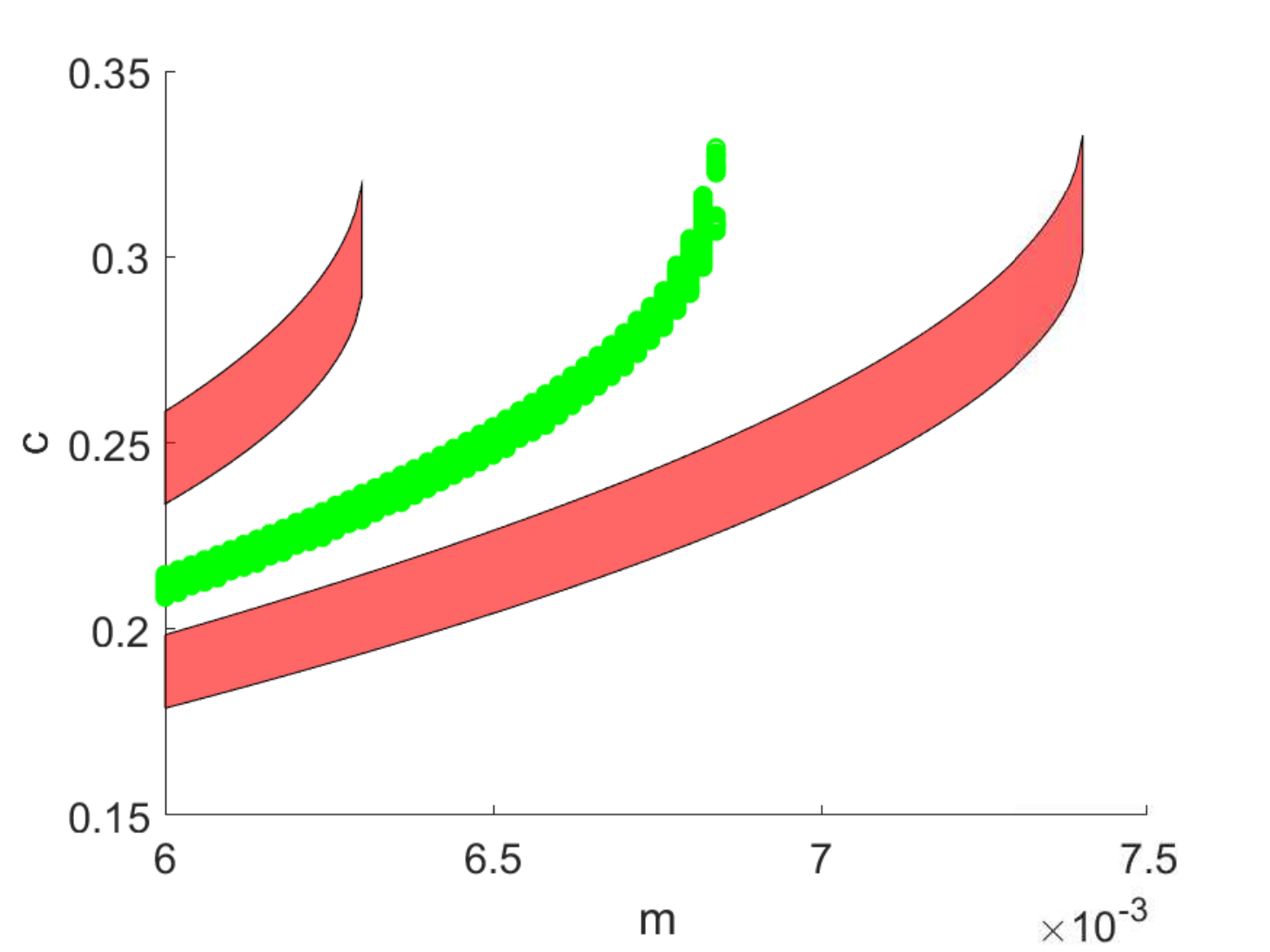}

\caption{On the left are locking regions for various speeds with $r=1.5$.  The red regions are locking regions corresponding to speeds $1/q$ with $q$ from $1$ to $6$.  The green regions are locking regions for speeds $2/q$ with $q$ from $3$ to $11$ with $q$ odd.  The blue regions are locking region for  speeds $3/q$ with $q$ from $4$ to $17$ with $\mathrm{gcd}(3,q)=1$.  On the right is the case of $r=1.1$.  Shown in red are locking regions with speed $1/19$ and $1/20$ calculated using $c_{\mathrm{max}}(r,m)$ and $c_{\mathrm{min}}(r,m)$ from (\ref{eq:cmaxgen}) and (\ref{eq:cmingen}). The green circles represent parameter values for which speed $2/39$ is observed.  At these values, direct numerical simulations of (\ref{eq:main}) are observed to propagate exactly $10,000$ lattice sites in $195,000$ iterations, after a transient of $100,000$ iterations is neglected.    }
\label{fig:ltscaling}
\end{figure}

One question considered in \cite{wang19} concerns the proportion of parameter space taken up by locked fronts, pulled fronts, and pushed (but not locked) fronts.  In \cite{wang19}, such estimates are derived using direct numerical simulations.  We had hoped that our approach could corroborate their findings, but the fact that small $p$ locking regions have relatively large measure makes this problematic.  For example, numerically computing the $s=2/39$ locking region requires obtaining the $37$ smallest roots of a degree $78$ polynomial and then solving (\ref{eq:vander}) to determine the constants $k_j$.   Our numerical routine was unable to determine reliable boundaries in this case using (\ref{eq:cmaxgen})-(\ref{eq:cmingen}).  Determination of the locking region using direct numerical simulation reveals that for some parameters, this locking region has significant size compared to other locking regions with smaller $q$ values; see Figure~\ref{fig:ltscaling}.

\section*{Acknowledgments}  This project was conducted as part of a year-long undergraduate research program hosted by the Mason Experimental Geometry Lab (MEGL).  The research of M.H. was partially supported by the National Science Foundation (DMS-2007759).  The authors thank the anonymous referees for comments that improved the paper.

\bibliographystyle{abbrv}
\bibliography{DTDSbib}

\end{document}